\documentclass{article}

\usepackage{arxiv}

\usepackage[utf8]{inputenc}
\usepackage[english]{babel}
\usepackage{amssymb,amsmath,amsthm}
\usepackage{mathtools} 

\usepackage[T1]{fontenc}    
\usepackage{hyperref}       
\usepackage{url}            
\usepackage{booktabs}       
\usepackage{amsfonts}       
\usepackage{nicefrac}       
\usepackage{microtype}      
\usepackage{lipsum}
\usepackage{graphicx}
\usepackage{mathrsfs}
\graphicspath{ {./images/} }
\usepackage[figuresright]{rotating}
\usepackage{parskip}
\usepackage{booktabs}
\usepackage{enumitem}
\usepackage{tikz}
\usetikzlibrary{calc}
\usepackage{caption}
\usepackage{subcaption}
\usepackage{longtable}
\usepackage{comment}
\usepackage{adjustbox}
\usepackage{mathtools}
\usepackage{moresize}
\usepackage{adjustbox}
\usepackage[ruled,vlined,linesnumbered]{algorithm2e}
\usepackage{nicematrix}
\usepackage{multirow}

\newcommand{\tL}{\texttt{L}}
\newcommand{\x}{\mathbf{x}}
\newcommand{\tM}{\texttt{M}} 
\newcommand{\M}{\mathbf{M}}
\newcommand{\QO}{\mathcal{Q}} 
\newcommand{\bfrak}{\mathfrak{b}}
\newcommand{\R}{\mathbb{R}}
\newcommand{\Z}{\mathbb{Z}}
\newcommand{\blambda}{\mathbf{\lambda}}
\newcommand{\hp}{\hspace{-1.5cm}}
\newcommand{\SOS}{\textit{\textsf{SOS1}}\xspace}
\newcommand{\dsum}{\displaystyle\sum}

\renewcommand{\;}{\quad}

\newtheorem{proposition}{Proposition}
\newtheorem{lemma}{Lemma}

\newtheorem{example}{Example}

\newtheorem{theorem}{Theorem}
\newtheorem{rem}{Remark}
\newtheorem{definition}{Definition}

\title{Linear, nested, and quadratic ordered measures: Computation and incorporation into optimization problems}

\author{
  Victor Blanco \\
  Department of Quantitative Methods, Faculty of Economics and Business Administration, University of Granada, Spain \\
  Institute of Mathematics of the University of Granada, Spain \\
  \texttt{vblanco@ugr.es} \\
  \And
  Miguel A. Pozo \\
  Department of Statistics and Operational Research. Faculty of Mathematics, University of Seville, Spain \\
  Institute of Mathematics of the University of Seville, Spain \\
  \texttt{puerto@us.es} \\
  \And
  Justo Puerto \\
  Department of Statistics and Operational Research. Faculty of Mathematics, University of Seville, Spain \\
  Institute of Mathematics of the University of Seville, Spain \\
  \texttt{puerto@us.es} \\
  \And
  Alberto Torrejon \\
  Department of Statistics and Operational Research. Faculty of Mathematics, University of Seville, Spain \\
  Institute of Mathematics of the University of Seville, Spain \\
  \texttt{atorrejon@us.es} \\
}

\begin{document}
\maketitle
\begin{abstract}
In this paper we address a unified mathematical optimization framework to compute a wide range of measures used in most operations research and data science contexts. The goal is to embed such metrics within general optimization models allowing their efficient computation.
We assess the usefulness of this approach applying it to three different families of measures, namely linear, nested, and quadratic ordered measures.
Computational results are reported showing the efficiency and accuracy of our methods as compared with standard implementations in numerical software packages. 
Finally, we illustrate this methodology by computing a number of optimal solutions with respect to different metrics on three well-known linear and combinatorial optimization problems: scenario analysis in linear programming, the traveling salesman and the weighted multicover set problem.
\end{abstract}


\section{Introduction}
\label{sec:introduction}

Mathematical optimization models are highly dependent on the choice of the measure to be maximized or minimized in the optimization process. Most of the models in the literature, as those that arise in facility location, network design, or even in machine learning, require to measure the goodness of a solution by aggregating into a single value the separated goodness measures for the different entities involved in the problem. For instance, in facility location, the transportation costs from each customer to a set of open facilities, the length/costs of each of the edges/arcs in a network design problems, the missclassification errors for each of the observations in a supervised classification problem, or the returns of the different items selected in a portfolio selection problem. 

Classical models assume that this aggregation is performed either by the average (sum) or the worst case (max or min) behavior. These two approaches have given rise to a vast applied optimization literature in different fields.  However, it has been largely recognized that many more point of views make sense in different frameworks. In this line, different aggregation measures have been proposed to derive solutions for decisions problems that allow the decision-maker to select alternatives not only focused on economically efficient returns, but also in fair allocations, robust criteria in uncertain situations, portfolio selection problems under risk, balanced criteria in supply chain management, envy free measures in social choice and distribution problems.

The choice of measures to summarize a list of values of a feature is a recurrent problem in Statistics, and one can find in the literature diverse attempts to propose families of statistics (or aggregation criteria) to accurately summarize data. This is the case of \tL- and \tM-statistics, tendency, dispersion, and shape measures in data science, OWA operators, ordered median functions, robust measures (as, for instance, regret, $k$-sums, etc.) or fairness measures (as equity or envy-free criteria). Although, all these measures are provided with an explicit expression for its computation, in some cases the application of iterative procedures is required, as it is the case of \tL-statistics that need sorting the data. Thus, to determine the value of these measures for a large given list of data, one has to resort to algorithms that perform the required operations in a reasonable amount of time. 

This paper has two different main goals. On the one hand, we derive a unified mathematical optimization framework to efficiently compute a wide family of metrics of interest for summarizing the information of a feature based on robust measures. More specifically, we prove for the first time, that the computation of any $\tL$-statistics applied to a given set of values is equivalent to solve a linear optimization problem. 
Moreover, we also propose a generalized family of quadratic measures, that having as particular case the variance of a set of values, allows for the introduction of novel robust dispersion measures that have not been previously introduced in the literature. The second part of the paper is focused on taking advantage of these formulations to integrate these measures, that have been already justified as useful in different areas of decision analysis, as objective functions within different challenging optimization problems. That is, when the lists of values to be summarized are inherent to the decision problems.

The optimization-based approach that we propose has three main advantages: 1) the provided representations of these measures are simply modeled as linear objectives in contrast to other complex representations in the literature; 2) they can flexibly accommodate different measures within an unified framework, which allows for a common algebraic structure of the problems facilitating their analysis and strengthening; and 3) they allow for common solution approaches simplifying their use for practitioners.

Our aim is to cover the most widely used  operators in Operations Research, including among others  minimax problems \cite{Hansen1980,Schrijver1982}, combining minisum and minimax \cite{AVERBAKH1995,HL88,HLT91,Punnen1995,TPP}, $k$-centrum optimization
\cite{Garfinkel,Kalcsics2002,Punnen92,Slater1978CentersTC,Tamir01}, lexicographic
optimization \cite{CalveteMateo98,Croce99}, $k$-th best solutions
\cite{Lawler72,MPTW84,PCC03,Yen71}, most uniform solutions \cite{Galil88,LMP08}, minimum-envy solutions \cite{Espejo2009,Blanco2024}, solutions with minimum deviation \cite{Gupta90}, regret solutions
\cite{Aver01,PRC03SM}, equity measures \cite{GP88,LMP08,MPT,PA97},
discrete ordered median location problems \cite{Boland2006,MNPV08,Marin2020,Ljubic2024},
discrete optimization with ordering \cite{Fernandez2013,Fernandez2014} or covering objectives \cite{Blanco2023fairness,BP72,Breuer70,CK74,Lawler66}; among many others.

The main and common ingredient of the general measures that we analyze in this paper is the need to order the values to be summarized to compute the final outcome.  It is well-known that introducing the order when defining criteria in a decision-making problem adds a higher level of complexity to the computation of these criteria, and then to obtain a solution of the problem. Thus, one of the main challenges when addressing ordered optimization problems is the issue of the computation of those criteria (measures) as mathematical programs so that they can be embedded into more complex structures where the use of such index is instrumental. This is the case of computing the location of a facility with minimum variance of the distances, a portfolio with minimum Conditional Value at Risk (CVaR), or linear model estimators with minimum deviation residuals with respect their mean. Some of these measures have been already recognized as useful to guide, in the decision-making problem, towards  fair or equitable solutions without compromising efficiency.

The remainder of this paper is structured as follows.
In Section \ref{sec:2}, we review the existing literature on the measures that can be generalized using our proposed formulations, along with their ordered representations.
In Section \ref{sec:3}, we present linear formulations for computing ordered measures, such as \tL-statistics, and introduce additional formulations to compute measures that incorporate a nested linear structure, for example, the variance, skewness or kurtosis coefficients.
In Section \ref{sec:4}, we propose a novel quadratic ordered operator that enables the generalization of various families of quadratic measures while extending to new ones.
In Section \ref{sec:5}, we conduct an empirical comparison of different solution methods.
Finally, in Section \ref{sec:6}, we discuss the practical implications of these approaches and their potential integration into complex optimization-based decision-making problems.

{\bf Contributions:}
\begin{itemize}[nosep]
    \item We provide a unified optimization-based framework to compute a wide range of operators, including $\tL$-statistics among them, by means of linear or bilevel linear optimization problems.
    \item We introduce a new family of measures that combine well-known ordered-type statistics with quadratic measures, and develop mixed integer linear optimization models for its computation.
    \item We evaluate the computational efficiency of our models by testing various formulations while computing multiple metrics across different instance sizes.
    \item We integrate the developed optimization models into more complex optimization problems where the values to summarize are part of the decision problem.
\end{itemize}

\section{Related Literature: Descriptive measures of tendency, dispersion and scale.}
\label{sec:2}

Given a sample of $n$ observations on a certain variable, namely $\x \in \R^n$, descriptive measures (at times called \textit{statistics}) are scalars that allow to summarize and describe the essential characteristics of the variable. The measure is chosen based on the feature that is intended to be analyzed. Among the large number of proposals, the most popular are: the \emph{measures of tendency}, that can be central (as the mean, median, mode) or non-central (as the quantiles, maximum, minimum); the \emph{measures of dispersion} (as the variance, standard deviation, range); and the \emph{measures of shape} (as the skewness and the kurtosis). The first family, namely, the measures of tendency, are usually the basis to define other types of indices (e.g, the skewness is defined as the asymmetry of the sample around the mean). 
In what follows we concentrate on two popular families of descriptive measures \tL-measures and \tM-measures (referred to as \tL-statistics and \tM-statistics in Statistics or Data Analysis).

A general framework for defining and studying the behavior and properties of many of these measures is to express them as a weighted linear combination of ordered values, which are known as \tL-measures. \tL-measures were first outlined by \cite{Daniell1920}, but practically considered as an useful framework in Robust Statistics by \cite{Huber2009}, since 
some of them have been proven to be statistically robust under the presence of outliers (whereas others, as the  minimum, maximum, mean, and mid-range are known to be sensitive to extreme values). 
For a complete definition of \tL-measures in Statistics, the interested reader is referred to \cite{Serfling2009} or \cite{Huber2009}. 

In what follows we provide a formal definition of this family of measures.

\begin{definition}
Let $\x \in \R^n$ be a sample and $\blambda \in \R^n$ a vector of weights. A \tL-measure for the sample $\x$ is given by:
\begin{equation}
\label{eq:Lestimator}
\ell(\x, \blambda) := \sum_{k=1}^n \lambda_k x_{(k)},
\end{equation}
where $x_{(i)} \in \{x_1, \ldots, x_n\}$ with  $x_{(1)} \leq x_{(2)} \leq \ldots \leq x_{(n)}$. 
\end{definition}
Note that every choice of the $\blambda$-weights results in a different operator. In Table \ref{table:location_lambdas} we show some of the most popular $\blambda$-weights and their corresponding tendency measure.

Other well-known family of measures are the $\tM$-measures, which are defined as those minimizing a given loss function with respect to each of the values in $\x$. The most popular loss functions in this framework are  the Huber function \cite{Huber1964, Huber1973}, the Tukey’s Biweight function \cite{Rousseeuw1987}, the logistic function \cite{Coleman1980}, and the Talwar function \cite{Hinich1975}.


\begin{table}[ht]
\renewcommand{\arraystretch}{1.75}
\centering
\footnotesize
\resizebox{\textwidth}{!}{
\begin{tabular}{ccc}
\hline
\texttt{Statistic} & \texttt{Expression} & $\blambda$ \\
\hline
    {Arithmetic mean} &
    {$\displaystyle{\frac{1}{n} \sum_{i=1}^n x_i}$} & 
    $\frac{1}{n}(1,...,1)$ 
\\ \hline
    {Weighted mean} &
    {$\displaystyle{\frac{1}{\sum_{i=1}^n w_i} \sum_{i=1}^n w_i x_i}, \quad w_i\geq 0, \sum_{i=1}^n w_i = 1$} & 
    $(\frac{w_1}{\sum_{i=1}^n w_i},...,\frac{w_n}{\sum_{i=1}^n w_i})$ 
\\ \hline
    {Maximum} &
    $\displaystyle{\max x_i} $ & 
    $(0,...,0,1)$ 
\\ \hline
    {Minimum} &
    $\displaystyle{\min x_i} $ & 
    $(1,0,...,0,0)$ 
\\ \hline
    {Median}  &
    \footnotesize $\displaystyle{\text{median} (x) } = \begin{cases}
        x_{(m+1)}, & n = 2m+1 \\
        \frac{x_{(m)} + x_{(m+1)}}{2}, & n = 2m
    \end{cases}$ & 
    \small{\begin{tabular}[c]{@{}c@{}}   
        $(0, ..., 0, \underbrace{1}_{(m+1)-th}, 0, ..., 0)$  \\ 
        $(0, ..., 0, \underbrace{\frac{1}{2}}_{m-th}, \underbrace{\frac{1}{2}}_{(m+1)-th}, 0, ..., 0)$ 
    \end{tabular}}
\\ \hline
    {\begin{tabular}[c]{@{}c@{}} $q$-th quantile \\ (or $\tau$-quantile) \end{tabular}} &
    $x_{(q)}$, $q = \lfloor \tau  n \rfloor$  & 
    $(0, ..., 0, \underbrace{1}_{q-th}, 0, ..., 0)$ 
\\ \hline
    {Midrange} &
    $\frac{1}{2} (\max x_i + \min x_i )$  & 
    $\frac{1}{2} (1, 0, ..., 0, 1)$ 
\\ \hline
    {Tukey's trimean} &
    $\frac{1}{4}(x_{(\lfloor n/4 \rfloor)} + 2 x_{(\lfloor n/2 \rfloor)} + x_{(\lfloor 3n/4 \rfloor)})$  & 
    $\frac{1}{4}(0, ..., 0, \underbrace{1}_{\lfloor n/4 \rfloor-th}, 0, ..., 0, \underbrace{2}_{\lfloor n/2 \rfloor-th}, 0, ..., 0, \underbrace{1}_{\lfloor 3n/4 \rfloor-th}, 0, ..., 0)$ 
\\ \hline
    {\begin{tabular}[c]{@{}c@{}} General midrange \\ (midhinge, etc.) \end{tabular}} &
    $\frac{1}{2}(x_{(r)} + x_{(s)})$, $(r \leq s)$ & 
    \footnotesize  $\frac{1}{2} (0, ..., 0, \underbrace{1}_{r-th}, 0, ..., 0, \underbrace{1}_{s-th} ..., 0)$ 
\\ \hline
    {\begin{tabular}[c]{@{}c@{}} $(r, s)$-trimmed mean \end{tabular}} &
    $\displaystyle{\frac{1}{n - r - s} \sum_{i=r+1}^{n-s} x_{(i)} }$ & 
    $\frac{1}{n - r - s} (\underbrace{0,...,0}_{r}, \underbrace{ 1,...,1, \ }_{n - r - s} \underbrace{0,...,0}_{s})$ 
\\ \hline
    {\begin{tabular}[c]{@{}c@{}} $(r,s)$-mid mean  \end{tabular}} &
    $\displaystyle{ \frac{1}{ r }\sum_{i= 1}^{r} x_{(i)} + \frac{1}{s} \sum_{i= n- s}^{n} x_{(i)} }$ & 
    $(\underbrace{\frac{1}{r} , ..., \frac{1}{r}}_{r}, \underbrace{0,...,0, \ }_{n-r-s} \underbrace{\frac{1}{s}, ..., \frac{1}{s}}_{s})$   
\\ \hline
    {\begin{tabular}[c]{@{}c@{}} $(r,s)$-winsorized mean \end{tabular}} &
    $\displaystyle{ \frac{1}{n} \Big( r x_{(r+1)} + \sum_{i=r+1}^{n-s} x_{(i)} + s x_{(n-s)} \Big) }$ & 
    $\frac{1}{n} (0,...,0, \underbrace{r+1}_{(r+1)-th}, 1,..., 1, \underbrace{s+1}_{(n-s)-th}, 0,...,0)$
\\ \hline
\end{tabular}
}
\caption{Location or tendency measures and $\lambda$-vectors.}
\label{table:location_lambdas}
\end{table}


On the other hand, dispersion measures quantify the extent to which the given values are spread. These measures use to take value zero when all data are identical and increase as the data become more diverse. 
Dispersion is usually assessed using as basis certain tendency-based measures and very often are sustained on measuring the  deviation from this reference. This is the case of the variance (as the squared deviation with respect to the mean), the standard deviation, or the mean absolute deviation that are computed as an aggregation of the difference of each value in the sample with respect to its mean. Robust measures of dispersion, additionally, aim at being resistant to the influence of outliers. Common robust dispersion measures include the interquartile range and the median absolute deviation. Dispersion measures allow to quantify abstract concepts such as fairness or equity, by measuring the deviations from the sample values to a center \emph{reference} point~\cite{Hoover1941}, or all pairwise deviations~\cite{Gini1921}. For an  overview on fairness or equity measures, the interested reader is referred to \cite{Marsh1994,sen1997economic,young1995equity,bertsimas2011price}.

The notion of \tL-measure introduced above can be naturally extended to derive dispersion measures~\cite{Welsh1990,Huber2009}. 
Furthermore, numerous measures of scale can be directly expressed as a \tL-measure in the shape of (\ref{eq:Lestimator}) by adequately choosing the $\blambda$-vector, as the interquartile range, Gini's mean difference, or efficient scale estimators as those proposed by \cite{Serfling2009}.
In Table \ref{table:dispersion_lambdas} we list some of the most popular dispersion indices that can be cast as \tL-measures or combinations of them, along with the $\blambda$ weights that must be chosen. 
Numerous additional families of measures of scale can be addressed as these can be constructed by combining other families, for instance, as weighted, trimmed, or winsorized versions of both squared and absolute deviations or pairwise differences (e.g., trimmed or winsorized variance).

\begin{table}[ht]
\centering
\footnotesize
\renewcommand{\arraystretch}{1.75}
\begin{tabular}{ccc}
\hline
\texttt{Statistic} & \texttt{Expression} & $\blambda$ \\
\hline
  \small{Range} &
  \footnotesize $\displaystyle{\max x_i - \min x_i}$ & 
  $(-1, 0, ..., 0, 1)$ 
\\ \hline
  \small{\begin{tabular}[c]{@{}c@{}} General interquantile range \\ (e.g., interquartile or interdecile ranges) \end{tabular}} &
  \footnotesize $\displaystyle{x_{(r)} - x_{(s)}, (r < s)}$ & 
  $(0, ..., 0, \underbrace{-1}_{r-th}, 0, ..., 0, \underbrace{1}_{s-th} ..., 0)$ 
\\ \hline
    \small{\begin{tabular}[c]{@{}c@{}} Variance \\ (as squared deviation \\ [-1em] from the mean) \end{tabular}} &
    \footnotesize $\displaystyle{\frac{1}{n} \sum_{i=1}^n \big(x_i - \frac{1}{n} \sum_{j=1}^n x_j \big)^2}$ & 
\\ \hline
    \small{\begin{tabular}[c]{@{}c@{}} Mean squared deviation \\ [-1em] from a location measure $m(x)$ \\ (e.g., from the $q$-th quantile) \end{tabular}} &
    \footnotesize $\displaystyle{\frac{1}{n} \sum_{i=1}^n \big(x_i - m(x) \big)^2}$ & 
\\ \hline
    \small{\begin{tabular}[c]{@{}c@{}} Gini differences or \\ [-1em] mean absolute difference \end{tabular}} &
    \footnotesize $\displaystyle{ \frac{1}{n^2} \sum_{i,j=1}^n \left|x_i-x_j\right| }$  &
    $\frac{2}{n^2}((2i-n-1))_{i \in \{1,...,n\}}$ 
\\ \hline
    \small{\begin{tabular}[c]{@{}c@{}} Mean absolute deviation \\ [-1em] from mean \end{tabular}} &
    \footnotesize $\displaystyle{\frac{1}{n} \sum_{i=1}^n \big| x_i - \frac{1}{n} \sum_{j=1}^n x_j \big|}$ & 
\\ \hline
    \small{\begin{tabular}[c]{@{}c@{}} Mean absolute deviation from \\ [-1em] a location measure $m(x)$ \\ (e.g., from the $q$-th quantile) \end{tabular}} &
    \footnotesize $\displaystyle{\frac{1}{n} \sum_{i=1}^n \big| x_i - m(x) \big|}$ & 
    {\begin{tabular}[c]{@{}c@{}} $\frac{1}{n}(-1, ..., -1, \underbrace{2q-n-1}_{q-th}, 1, ..., 1)$  \\ (e.g., $m(x) = x_{(q)}$)  \end{tabular}}
%
\\ \hline
    \small{\begin{tabular}[c]{@{}c@{}} Median absolute deviation from \\ [-1em] a location measure $m(x)$ \\ (e.g., from the mean) \end{tabular}} &
    \footnotesize $\displaystyle{ median(\big| x_i - m(x) \big| )}$ & 
\\ \hline
\end{tabular}
\caption{Dispersion or scale measures and $\lambda$-vectors.}
\label{table:dispersion_lambdas}
\end{table}




The shape of the \tL-measures presented above has also been applied in other contexts of Statistics, Decision Theory and Data Science, since they can be seen as aggregation operators for a set of values. Expressions in the form of (\ref{eq:Lestimator}) are known as Ordered Weighted Average (OWA) operators. OWA operators provide a parameterized class of mean-type aggregation operators introduced by \cite{Yager1988}. They are used as aggregation measures of unknown values to be determined by a decision problem, whereas in \tL-measures, as introduced above, the values are assumed to be known. In addition, they have been successfully applied to derive robust estimators in other areas of Statistics and Data Science. For instance, the aggregation of (generalized) residuals of linear regression models using OWA operators has given rise to the computation of robust estimators for these models in  the presence of outliers~\cite{Yager2009,  Blanco2018, Flores2020, Blanco2021, Flores2022, durso2022, Puerto2024}. Classification models through support vector machines with OWA aggregation of missclassification errors have also been recently considered~\cite{Maldonado2018, Marin2022}. Furthermore, OWA operators have been applied to the aggregation of expert's information in a Bayesian framework to derive insurance premiums~\cite{Blanco2018BMS}; to model the behavior of Neural networks with fuzzy quantifiers~\cite{Yager1987, Yager1992}, to image segmentation by grouping intensities~\cite{Calvino2022, Munoz2024} and to clustering problems \cite{Cheng2009}, among other fields.

Apart from all these applications of ordered aggregation operators mentioned above, these operators have been widely used in other disciplines of Operations Research. One of the most popular is the use in Facility and Hub Location problems as a unified methodology to cast different distance/cost-based objective functions~\cite{Nickel2006, Puerto2019, Labbe2017, Marin2020, Blanco2023fairness, Ljubic2024, Schnepper2019, Pozo2021, Pozo2024}. Other fields where OWA operators have been successfully applied include: voting problems~\cite{Amanatidis2015, Ponce2018}, portfolio selection~\cite{Cesarone2024} and network design \cite{Puerto2015OrderedMH,Pozo2021} and combinatorial optimization \cite{Fernandez2013,Fernandez2014}.

\section{Optimization based approaches to compute generalized measures}
\label{sec:3}

In this section we provide a general and unified framework based on mathematical optimization models to compute different families of operators based on $\tL$-measures by solving a linear optimization problem. As already mentioned, this modeling approach has several advantages. One of them is that it allows for the integration of the problem into more complex decision making problems, as we will show in later sections of this paper. Additionally, in some cases this approach is computationally competitive with respect to  the standard sorting-based approaches implemented in software packages, as we will see in our computational experiments. 

\subsection{$\tL$-measures as Linear Optimization Problems}
\label{sec:3.1}

Let $\x \in \R^n$ be a given vector, and denote by $N=\{1, \ldots, n\}$ the index set for the components of this vector. The goal of this section is to compute $\ell(\x, \blambda)$ for a given set of weights $\blambda \in \R^n$. Although it can be done by sorting the values of $\x$ in non increasing order and calculating the weighted sum of the resulting vector by $\blambda$, here we propose a method for computing its value by means of solving a mathematical optimization problem. Note that particular cases of these measures such as the mean, the maximum, or the minimum of $\x$ can be  \emph{easily} modeled as the following linear optimization problems:

\begin{minipage}{0.33\linewidth}
\begin{subequations}
\label{form:mean_linear}
\begin{align}
& \text{mean} (\x) =  \min \, \gamma                      \label{cons:mean_lp_0} \\
& \mbox{s.t.} \, \textstyle \dsum_{i\in N} x_i \leq n \gamma, \label{cons:mean_lp_1} \\
& \hspace{0.6cm} \gamma \in \mathbb{R}.                      \label{cons:mean_lp_2}
\end{align}
\end{subequations}
\end{minipage}~
\begin{minipage}{0.33\linewidth}
\begin{subequations}
\label{form:max_linear}
\begin{align}
& \text{max} (\x) = \min \, \gamma                    \label{cons:max_lp_0} \\
& \mbox{s.t.} \;  x_i \leq \gamma, \, \forall  i\in N, \label{cons:max_lp_1} \\
& \hspace{0.9cm}    \gamma \in \mathbb{R}.               \label{cons:max_lp_2}
\end{align}
\end{subequations}
\end{minipage}~
\begin{minipage}{0.33\linewidth}
\begin{subequations}
\label{form:min_linear}
\begin{align}
&  \text{min} (\x) = \max \quad \gamma                   \label{cons:min_lp_0} \\
&  \mbox{s.t.} \;    x_i \geq \gamma, \, \forall i\in N, \label{cons:min_lp_1} \\
&  \hspace{0.9cm}    \gamma \in \mathbb{R}.              \label{cons:min_lp_2}
\end{align}
\end{subequations}
\end{minipage}\\

Although the above models are straightforward to derive, the general case of $\tL$-measure is more complex. Indeed, optimization-based expressions for $\tL$-measures have been widely studied in the literature in different contexts. Before providing our approach we recall two differentiated formulations that have been previously proposed, and that can be applied to compute $\tL$-measures. 

The first model was introduced by \cite{Boland2006}, and consists of an assignment-based integer linear programming model.
\begin{proposition}
Let $\x, \blambda \in \R^n$. Then
\label{form:L_1}
\begin{subequations}
\begin{align}
\ell_1(\x,\blambda) &= &\min & \;\sum_{i,k=1}^n \lambda_k z_{ik} x_{i},                    & \label{cons:L_1_0} \\ 
&&\text{s.t. } & \sum_{i=1}^{n} z_{ik}=1,    & \forall k \in N,     \label{cons:L_1_1}\\
&&& \sum_{k=1}^{n} z_{ik}=1,         & \forall i \in N,      \label{cons:L_1_2} \\
&&& \sum_{i=1}^{n} z_{ik} x_{i} \leq \sum_{i=1}^{n} z_{i,k+1} x_{i}, & \forall k \in N, k<n, \label{cons:L_1_3} \\
&&& z_{ik} \in\{0,1\},     & \forall i, k \in N.   \label{cons:L_1_4} 
\end{align}
\end{subequations}
\end{proposition}
\begin{proof}
The proof follows by identifying the $z$-variables with the assignment variables that determine the order in the sorting sequence for the values in the sample $\x$:
$$z_{ik} = \begin{cases}1, \text{ if the $i$-th element is in the $k$-th position of the ordered vector}, \\ 0, \text { otherwise, }\end{cases} \forall i, k \in N.$$
These variables are properly defined through constraints \eqref{cons:L_1_1} and \eqref{cons:L_1_2} which ensure that each component of $\x$ is allocated to exactly one position and each position is assigned to exactly one component of $\x$, and \eqref{cons:L_1_3} which enforce that the components after the sorting are in non increasing order. With these variables $x_{(k)} 
= \sum_{i\in N} z_{ik} x_i$ indicates the $k$th component in the sorted vector, and then the objective function is exactly the expression for the $\tL$-measure. 
\end{proof}
Note that this model is an integer linear program with $\mathcal{O}(n^2)$ variables, which may result in computational difficulties when applied to large datasets.

The second model to compute $\ell(\x,\blambda)$ via an optimization problem is based on a formulation provided in \cite{Marin2020} for the  Discrete Ordered Median Location Problem (DOMP) which relies on the $k$-sum representation of the OWA operator. This model is a linear optimization problem that avoids the use of binary variables, which provides a great computational advantage over the previous model.

\hfill

\begin{proposition}
\label{prop:L_ksum}

Let $\x, \blambda \in \R^n$. Then
\begin{subequations}
\label{form:L_2}
\begin{align}
\ell_2(\x,\blambda) &=&\min & \sum_{k \in \Delta^{+}} \Delta_{k} \big((n-k+1) t_{k}+\sum_{i=1}^n v_{ik} \big) + \sum_{k \in \Delta^{-}} \Delta_{k} \sum_{i=1}^n d_{ik} x_{i},  & \label{cons:L_2_0} \hspace{-1cm} \\ 
&&\mbox{s.t. \ } & t_k+v_{ik} \geq x_i,                 & \forall i \in N, k \in \Delta^{+}, \label{cons:L_2_1} \\
&&& \sum_{i=1}^n d_{ik}=n-k+1,           & \forall k \in \Delta^{-},          \label{cons:L_2_2} \\
&&& t_{k} \in \mathbb{R}, v_{ik} \geq 0, & \forall i \in N, k \in \Delta^{+}, \label{cons:L_2_3} \\
&&& 0 \le d_{ik} \leq 1,                 & \forall i \in N, k \in \Delta^{-}. \label{cons:L_2_4}
\end{align}
\end{subequations}
where $\lambda_{0}:=0$ and $\Delta_{k}:=\lambda_{k}-\lambda_{k-1}$ for all $k \in \{1,...,n\}$, $\Delta^{-} = \{k \in \{1,...,n\} | \Delta_k < 0 \}$, and $\Delta^{+} = \{k \in \{1,...,n\} | \Delta_k > 0 \}$.
\end{proposition}

\begin{proof}
    Let us define the $k$-sum operator as the sum of the largest $n-k+1$ elements of $\x$, that is:
\begin{equation}
\label{eq:ksums}
S_k(x)=\sum_{\ell=k}^n x_{(\ell)},  
\end{equation}
Using this operator, the expression of the $\tL$-measure can be rewritten as: 
\begin{equation} 
\label{eq:delta}
    \ell(\x,\blambda) = \sum_{k=1}^n \lambda_{k} x_{(k)} = \sum_{k=1}^n \Delta_{k} S_{k}(x),
\end{equation}

Now, observe that computing (\ref{eq:ksums}) is equivalent to solve the following linear problem or its dual  (see e.g.  \cite{Ogryczak2003,Kalcsics2002}):

\begin{minipage}{0.4\linewidth}
\begin{equation}
\begin{aligned}
& \max \sum_{i=1}^n d_i x_i,   \\ \mbox{s.t. \ }
& \sum_{i=1}^n d_i=n-k+1,      \\
& 0 \leq d_i \leq 1, \; \forall i \in N.
\end{aligned}
\end{equation}
\end{minipage}~
\begin{minipage}{0.5\linewidth}
\begin{equation}
\begin{aligned}
& \min \ (n-k+1) t+\sum_{i=1}^n v_i, \\ \mbox{s.t. } 
& t + v_i \geq x_i, \; \forall i \in N, \\
& t\in \mathbb{R}, v_i \geq 0, \; \forall i \in N.
\end{aligned}
\end{equation}
\end{minipage}

where in case $\x \in \R^n_+$, the variable $t$ can be considered nonnegative, namely $t\ge 0$. Therefore, given (\ref{eq:ksums}) and these primal-dual relationships of the $k$-sum operator, they can be combined to derive the proposed model.
\end{proof}

Observe that the model above is a linear optimization model with $O(n^2)$ variables and $O(n^2)$ constraints. 

Finally, in what follows we derive a novel model that we propose to compute $\tL$-measures, as another alternative linear optimization problem. The main ingredient of this model is that it is derived from the computation of the $\tau$-quantiles of $\x$. 

For a given $0<\tau<1$, we define:
\begin{equation}
\label{eq:quantile_formula}    
\text{quantile}_\tau (x) = \arg \min_{\gamma \in \mathbb{R}} \mathcal{L}_{\tau}(x,\gamma),
\end{equation}
where
\begin{equation}
\label{eq:tau_sample_quantile_function}
\mathcal{L}_{\tau}(x,\gamma) = \tau \sum_{x_i \geq \gamma} (x_i-\gamma) +  (1-\tau) \sum_{x_i<\gamma} (\gamma-x_i),
\end{equation}

The reformulation of the $\tau$-quantile problem above as the linear optimization problem presented in \cite{Puerto2024} allows for the alternative linear optimization model described in the following result.

\hfill

\begin{theorem}
\label{theo:L_q}
Let $\x, \blambda \in \R^n$. Then:
\begin{subequations}
\label{form:L_3}
\begin{align}
\ell_3(\x, \blambda) = \min & \sum_{k \in \Delta^{+}} \sum_{i\in N} \Delta_{k} \Big( \frac{k}{n} u_{ik} + (1-\frac{k}{n}) v_{ik} \Big) + \sum_{k \in \Delta^{-}} \sum_{i\in N}  \Delta_{k} \alpha_{ik} x_{i} + \sum_{k \in N} \sum_{i\in N} \Delta_{k} (1-\frac{k}{n}) x_i,& \label{cons:L_3_0}  \\ 
\mbox{s.t. \ } 
& u_{ik} - v_{ik} + \gamma_k = x_i,  \;\; \forall i \in N, k \in \Delta^{+}, \label{cons:L_3_1} \\
& \sum_{i\in N} \alpha_{ik}=0,  \;\; \forall k \in \Delta^{-},          \label{cons:L_3_2} \\
& \gamma_k \in \mathbb{R}, u_{ik}, v_{ik} \geq 0,   \;\;  \forall i \in N, k \in \Delta^{+}, \label{cons:L_3_3} \\
& -\frac{k}{n} \leq \alpha_{ik} \leq 1-\frac{k}{n}, \;\;  \forall i \in N, k \in \Delta^{-}, \label{cons:L_3_4}
\end{align}
\end{subequations}
where $\lambda_{0}:=0$ and $\Delta_{k}:=\lambda_{k}-\lambda_{k-1}$ for all $k \in \{1,...,n\}$, $\Delta^{-} = \{k \in \{1,...,n\} | \Delta_k < 0 \}$, and $\Delta^{+} = \{k \in \{1,...,n\} | \Delta_k > 0 \}$.
\end{theorem}

\begin{proof}
Note first that the $\tau$-quantile problem can be solved as the following linear optimization problem:
\begin{subequations}
\label{form:quant_lp}
\begin{align}
\min &\; \tau \sum_{i=1}^n u_i + (1-\tau) \sum_{i=1}^n v_i,  &   \label{cons:quant_lp_0} \\
\mbox{s.t.\ } & u_i -v_i + \gamma = x_i,  & \forall i \in N,  \label{cons:quant_lp_1} \\
& u_i, v_i \geq 0,                        & \forall i \in N,  \label{cons:quant_lp_2} \\
& \gamma \in \mathbb{R}.                                      \label{cons:quant_lp_3} 
\end{align}
\end{subequations}

Its dual is the following problem:
\begin{subequations}
\label{form:quantile_linear_dual}
\begin{align}
\max & \sum_{i=1}^n \alpha_i x_i, &                  \label{cons:quant_lp_dual_0} \\ \text{s.t.: \ }
& \sum_{i=1}^n \alpha_i = 0,      &                   \label{cons:quant_lp_dual_1} \\
& \tau-1 \leq \alpha_i \leq \tau, & \forall i \in N. \label{cons:quant_lp_dual_2} 
\end{align}
\end{subequations}

Next, by Theorem 1 in \cite{Puerto2024}, we get that for $q=\tau n \in \Z_+$ the following equation holds:
$$\sum_{k=1}^n \lambda_{k} x_{(k)} = \sum_{k=1}^n \Delta_{k} S_{k}(x) = \sum_{k\in \Delta^+} \Delta_{k}  \left({\min_{\gamma_k} \mathcal{L}_{\tau}(x, \gamma_k)} + (1-\tau) \sum_{i = 1}^n x_i\right) + \sum_{k \in \Delta^-} \Delta_{k}  \left({\min_{\gamma_k} \mathcal{L}_{\tau}(x, \gamma_k)} + (1-\tau) \sum_{i = 1}^n x_i\right),$$
where $\tau = k/n$.

Using the above comments, we can decompose the expression $\ell(\x,\blambda)$ into two parts to obtain the objective function \eqref{cons:L_3_0}.  On the one hand, for the term summing up those $k \in \Delta^+$ we use the primal representation of the $\tau$-quantile problem \eqref{cons:quant_lp_0}-\eqref{cons:quant_lp_3}. In the second term ($k \in \Delta^-$) we use the dual representation \eqref{cons:quant_lp_dual_0}-\eqref{cons:quant_lp_dual_2}. 
\end{proof}

Note that formulations \eqref{form:L_2} and \eqref{form:L_3} can both be decomposed as two independent linear optimization problems that can be solved independently, and then to obtain the value of $\ell(\x,\blambda)$ as the sum of the objectives values of both problems, namely:
$$
\ell_i(\x,\blambda) = \ell_i^+(\x,\blambda) - \ell_i^-(\x,\blambda) + c_i, 
$$

for $i \in \{2,3\}$, 
and where $\ell^+_i$ is given by the part of the objective function and the constraints concerning the $\Delta^+$ set, 
$\ell^-_i$ is given by the part of the objective function and the constraints concerning the $\Delta^-$ set, 
and $c_i$ is a constant term that must be added, i.e., 
$c_2=0$ and $c_3 = \sum_{k \in \Delta_k^+} \sum_{i\in N} \Delta_{k} (1-\frac{k}{n}) x_i - \sum_{k \in \Delta_k^-} \sum_{i\in N} \Delta_{k} x_i$.
These linear problems have $O(n|\Delta^+|)$ and $O(n|\Delta^-|)$ variables, respectively.

\subsection{Nested $\tL$-measures as Bilevel Optimization Problems}
\label{sec:3.2}

In this section, we analyze some optimization models that allow us to minimize measures that require a nested $\tL$-measure, in their definition, i.e., given $\x, \blambda \in \R^n$ we aim to compute:
\begin{equation}
\label{eq:NestedMeasure}
\bfrak (\x, \rho) := \min  \rho \big(\x,  \sum_{k=1}^n \lambda_k x_{(k)} \big).
\end{equation}
for a given loss function $\rho$.

We will cast the calculation of this measure as a special class of optimization problems: bilevel optimization problem. A bilevel optimization problem (BOP) represents the relationship in a decision process involving hierarchical decision-making in which an upper-level decision maker aims to optimize its own objective function by selecting a decision that limits the lower-level decision's space, from which the lower-level decision maker needs to select its best decision in terms of its own objective function, see \cite{Bracken1973, Candler1977}.
In case the lower-level problem is a linear optimization program, the most frequent solution method for such BOPs in practice, and the one that we adopt in this section, is to reformulate the model as a single-level problem by means of the Karush-Kuhn-Tucker (KKT) conditions \cite{KKT1939, KKT1951} of the lower-level problem, that allow to incorporate the optimality conditions of this level as constraints in the upper level.  The interested reader is referred to the survey on exact methods for bilevel otimization by \cite{Kleinert2021} for further details.

First, let us rewrite the expression of  $\bfrak (\x, \rho)$ as a BOP. In the lower-level problem, the computation of the $\tL$-measure, $\ell(\x,\blambda) = \sum_{k=1}^n \lambda_k x_{(k)}$, is considered, and which will be represented by a nonnegative variable $\theta \geq 0$, that is
\begin{align}
  \label{cons:N_1}
  \bfrak (\x, \rho) = & \min \rho(\x, \theta) \\ \mbox{s.t. \ } 
  & \theta =  \ell(\x,\blambda) \label{theta:l}
\end{align}

Note that constraint \eqref{theta:l} can be represented as any of the optimization models described in the previous sections. 
Since our goal will be to reformulate this problem as a single level problem, we will focus on the two linear optimization formulations for its computation described in Proposition \ref{prop:L_ksum} and Theorem \ref{theo:L_q}.
Using the $k$-sum reformulation provided in Proposition \ref{prop:L_ksum} we obtain the reformulation detailed in the following result.
\begin{theorem}
\label{thr:nested-1}
Let $\x, \blambda \in \R^n$. Then:
\begin{subequations}
\label{form:N_1}
\begin{align}
\bfrak_1(\x, \rho) = & \min \ \rho(\x, \theta)  \label{cons:N_1_0} \\ \mbox{s.t. \ } 
& \theta = \sum_{k \in \Delta^{+}} |\Delta_{k}| \big((n-k+1) t_{k} + \sum_{i=1}^n v_{ik} \big) - \sum_{k \in \Delta^{-}} |\Delta_{k}| \sum_{i=1}^n d_{ik} x_{i},  \label{cons:N_1_1} \\ 
& \theta = - \sum_{k \in \Delta^{-}} |\Delta_{k}| \big((n-k+1) t_{k} + \sum_{i=1}^n v_{ik} \big) + \sum_{k \in \Delta^{+}} |\Delta_{k}| \sum_{i=1}^n d_{ik} x_{i}, \label{cons:N_1_2} \\ 
& t_k+v_{ik} \geq x_i,                 & \forall i,k \in N,     \label{cons:N_1_3} \\
& \sum_{i=1}^n d_{ik}=n-k+1,           & \forall k \in N,       \label{cons:N_1_4} \\
& t_{k} \in \mathbb{R}, v_{ik} \geq 0, & \forall i,k \in N,     \label{cons:N_1_5} \\
& 0 \le d_{ik} \leq 1,                 & \forall i,k \in N,     \label{cons:N_1_6} \\
& \theta \geq 0.                                                \label{cons:N_1_7}
\end{align}
\end{subequations}
\end{theorem}

\begin{proof}
Note that the dual problem of the $k$-sum problem \eqref{cons:L_2_0}-\eqref{cons:L_2_4} is:
\begin{subequations}
\begin{align}
& \max \ \sum_{k \in \Delta^+} \sum_{i=1}^n d^*_{ik} x_{i} + \sum_{k \in \Delta^-} ((n-k+1)t_k^* + \sum_{i=1}^n v^*_{ik}) \label{cons:Nd_1_0} \\ \mbox{s.t. \ } 
& \sum_{i=1}^n d^*_{ik} = (n-k+1) \Delta_k ,    & \forall k \in \Delta^+,       \label{cons:Nd_1_2} \\
& 0 \le d_{ik}^* \leq \Delta_k,                 & \forall i,k \in \Delta^+,     \label{cons:Nd_1_3} \\
& t^*_k+v^*_{ik} \leq x_i \Delta_k,             & \forall i,k \in \Delta^-,     \label{cons:Nd_1_4} \\
& t^*_{k} \in \mathbb{R}, v^*_{ik} \geq 0,      & \forall i,k \in \Delta^-.     \label{cons:Nd_1_5}
\end{align}
\end{subequations}
Then, by scaling the decision variables and rewriting:
\begin{equation}
    d_{ik}=\frac{d^*_{ik}}{|\Delta_k|}, i \in N, k \in \Delta^+, 
    t_k = - \frac{t_k^*}{|\Delta_k|}, v_{ik} = - \frac{v_{ik}^*}{|\Delta_k|}, i \in N, k \in \Delta^-, 
\end{equation}
we get that problem \eqref{cons:Nd_1_0}- \eqref{cons:Nd_1_5} is equivalent to:
\begin{subequations}
\begin{align}
& \max \ \sum_{k \in \Delta^+} |\Delta_k| \sum_{i=1}^n d_{ik} x_{i} - \sum_{k \in \Delta^-} |\Delta_k| ((n-k+1)t_k + \sum_{i=1}^n v_{ik}) \label{cons:NdE_1_0} \\ \mbox{s.t. \ } 
& \sum_{i=1}^n d_{ik} = (n-k+1),            & \forall k \in \Delta^+,     \label{cons:NdE_1_1} \\
& 0 \le d_{ik} \leq 1,                      & \forall i,k \in \Delta^+,   \label{cons:NdE_1_2} \\
& t_k+v_{ik} \geq x_i,                      & \forall i,k \in \Delta^-,   \label{cons:NdE_1_3} \\
& t_{k} \in \mathbb{R}, v_{ik} \geq 0,      & \forall i,k \in \Delta^-.   \label{cons:NdE_1_4}
\end{align}
\end{subequations}
Therefore, by considering the lower-level primal constraints \eqref{cons:L_2_1}-\eqref{cons:L_2_4} and dual constraints \eqref{cons:NdE_1_1}-\eqref{cons:NdE_1_4}, and both objective functions equivalence, we get the single-level formulation \eqref{cons:N_1_0}-\eqref{cons:N_1_7}
\end{proof}

For the $\tau$-quantile formulation (Theorem \ref{theo:L_q}), and using duality conditions similar to those in the previous result, we obtain the single-level reformulation described in the following result.
\begin{theorem}
\label{thr:nested-2}
Let $\x, \blambda \in \R^n$. Then:
\begin{subequations}
\label{form:N_2}
\begin{align}
\bfrak_2 (\x, \rho) = & \min \ \rho(\x, \theta)  \label{cons:N_2_0} \\ \mbox{s.t. \ } 
& \theta = \sum_{k \in \Delta^{+}} \sum_{i\in N} |\Delta_{k}| \Big( \frac{k}{n} u_{ik} + (1-\frac{k}{n}) v_{ik} \Big) + \sum_{k \in \Delta^{-}} \sum_{i\in N} |\Delta_{k}| \alpha_{ik} x_{i} \notag \\  
& \hspace{0.5cm} + \sum_{k \in N} \sum_{i\in N} \Delta_{k} (1-\frac{k}{n}) x_i, \label{cons:N_2_1} \\ 
& \theta = - \sum_{k \in \Delta^{-}} \sum_{i\in N} |\Delta_{k}| \Big( \frac{k}{n} u_{ik} + (1-\frac{k}{n}) v_{ik} \Big) - \sum_{k \in \Delta^{+}} \sum_{i\in N} |\Delta_{k}| \alpha_{ik} x_{i}  \notag \\ 
& \hspace{0.5cm} + \sum_{k \in N} \sum_{i\in N} \Delta_{k} (1-\frac{k}{n}) x_i, \label{cons:N_2_2} \\ 
& u_{ik} - v_{ik} + \gamma_k = x_i,                 & \forall i,k \in N,  \label{cons:N_2_3} \\
& \sum_{i\in N} \alpha_{ik}=0,                      & \forall k \in N,    \label{cons:N_2_4} \\
& \gamma_k \in \mathbb{R}, u_{ik}, v_{ik} \geq 0,   & \forall i, k \in N, \label{cons:N_2_5} \\
& -\frac{k}{n} \leq \alpha_{ik} \leq 1-\frac{k}{n}, & \forall i, k \in N, \label{cons:N_2_6} \\
& \theta \geq 0.                                                          \label{cons:N_2_7}
\end{align}
\end{subequations}
\end{theorem}

Note that the two proposed single-level optimization models to compute $\bfrak (\x, \rho)$ consist of linear constraints, and the non-linearity arises when $\rho$ is nonlinear. This is true even for the most popular loss functions, the absolute or quadratic deviations, or shape measures such as the skewness or kurtosis coefficients, as well as some of the already introduced functions.

In the following, we give an illustrative example detailing how to proceed in the important case where the loss function is the absolute deviation.

\hfill

\begin{example}
    Let us consider the loss function $\rho(\x,\theta)=\dsum_{i\in N}|\x_i-\theta|$ which induces the mean absolute deviation with respect to any  $\tL$-measure:
\begin{equation}
    MeanAD_{\blambda}(\x)  = \frac{1}{n} \sum_{i=1}^n \left|\x_i - \sum_{k=1}^n \lambda_k x_{(k)}\right|,
\end{equation}
Introducing variables $r_{i}^{+}, r_{i}^{-} \geq 0$, $i \in N$, to linearize the absolute values, we get that each of the addends in the above sum can be expressed as:
$$
\begin{aligned}
&  |x_i - \sum_{k=1}^n \lambda_k x_{(k)} | = r_{i}^{+}+r_{i}^{-}\\
& x_i - \sum_{k=1}^n \lambda_k x_{(k)} =  r_{i}^{+}-r_{i}^{-}, \\
& r_{i}^{+}, r_{i}^{-} \geq 0, r_{i}^{+} r_{i}^{-}=0, i \in N.
\end{aligned}
$$

Thus, we have that: 
\begin{subequations}
\label{form:bilevel_2}
\begin{align}
  MeanAD_{\blambda}(\x) = & \min \sum_{i=1}^n (r_{i}^{+} + r_{i}^{-}) &         \label{cons:BL_2_0} \\ \text{s.t.: } 
  & r_{i}^{+} - r_{i}^{-} = x_i - \theta,                          & i \in N,   \label{cons:BL_2_1} \\
  & r_{i}^{+}, r_{i}^{-} \geq 0,                                   & i \in N,   \label{cons:BL_2_3} \\
  & (r_{i}^{+}, r_{i}^{-}): \SOS,                                  & i \in N,   \label{cons:BL_2_4} \\
  & \theta = \sum_{k=1}^n \lambda_k x_{(k)}
\end{align}
\end{subequations}

where, \SOS  stands for the Special Ordered Sets of type 1, introduced by \cite{Beale1970}, which are a special set of variables for which at most one of them can take a non-zero value, while all other remain zero.
Then, following the same rationale as in Proposition \ref{prop:L_ksum} or Theorem \ref{theo:L_q}, one can rewrite the problem as a single-level linear optimization problem with \SOS constraints.
\end{example}

Finally, we would like to highlight that tailored models for cases where the nested measure (the $\tL$-measure) has some simplified form, as the maximum, the minimum, the mean or the quantile, give rise to simpler optimization models  (\eqref{form:max_linear}, \eqref{form:min_linear}, \eqref{form:mean_linear} or \eqref{form:quant_lp}, respectively). 

\section{A new family of quadratic measures, $\QO$-measures}
\label{sec:4}

In this section we introduce a new family of measures that extend the definition of $\tL$-measure by means of quadratic terms, and that include as particular cases the most popular measures of scale, as well as the linear $\tL$-measures.

\hfill 

\begin{definition}
\label{eq:QO_M}
Let $\x \in \mathbb{R}^n$, and a  matrix $\M = (m_{ij})_{i,j=1}^n \in \mathbb{R}^{n\times n}$, we define the quadratic ordered operator, namely $\QO$-measure, as:
\begin{equation}
\QO(\x, \M) = \sum_{i,j=1}^n m_{ij} x_{(i)} x_{(j)} 
\end{equation}
\end{definition}

Note that particular choices of $\M$ result in different measures which require quadratic terms. 
It is straightforward to check that $\QO(x, \M)$ is continuous,  symmetric, and (strictly) convex if and only if $\M$ is a (positive) semidefinite matrix. 
Quadratic terms allow us to measure interactions between the different elements of the vector, defined as product of the components, that is, one can consider the interaction between the maximum and the minimum value, $x_{(1)} x_{(n)},$
or the sum of the pairs of consecutive values, 
$x_{(1)} x_{(2)}(x) + x_{(2)} x_{(3)} + ... + x_{(n-1)} x_{(n)}.$
Note that, this approach can be generalized to include interactions of other orders by considering tensors instead of matrices.

Before analyzing the optimization-based approach that we propose to compute this type of operator, we provide some well-known interesting measures that can be rewritten as a $\QO$-measure.

\hfill

\begin{example}[Squared Deviation with respect to a $\tL$-measure]\label{ex:2}
Let us consider the nested $\tL$-measure where $\rho$ is the mean squared deviation function and $\ell (\x, \blambda) = \sum_{k=1}^n \lambda_k x_{(k)}$, this is:
$$\bfrak(\x,\rho) = \sum_{i = 1}^n \big( x_i - \ell (\x, \blambda) \big)^2.$$
Note that:
$$
\begin{aligned}
\sum_{i = 1}^n \big( x_i - \ell (\x, \blambda) \big)^2  
& = \sum_{i = 1}^n \big( x_i - \sum_{j = 1}^n \lambda_j x_{(j)} \big)^2 =  \sum_{i = 1}^n \big( x_{(i)} - \sum_{j = 1}^n \lambda_j x_{(j)} \big)^2 \\
& = \sum_{i=1}^n x_{(i)}^2 + \sum_{i=1}^n \Big( \sum_{j = 1}^n \lambda_j x_{(j)} \Big)^2 - \sum_{i=1}^n2 x_{(i)} \sum_{j=1}^n \lambda_j x_{(j)} \\
& = \sum_{i=1}^n x_{(i)}^2 + n \sum_{i=1}^n \lambda_{i}^2 x_{(i)}^2 + n \sum_{\substack{i,j=1\\ i\neq j} }^n \lambda_i \lambda_j x_{(i)} x_{(j)} - 2  \sum_{i=1}^n \lambda_{i} x_{(i)}^2 - 2 \sum_{\substack{i,j=1\\ i\neq j} }^n \lambda_j x_{(i)} x_{(j)} \\
& = \sum_{i=1}^n (1+n \lambda_i^2 - 2 \lambda_i) x_{(i)}^2 + \sum_{\substack{i,j=1\\ i\neq j} }^n (n\lambda_i \lambda_j - 2 \lambda_j) x_{(i)} x_{(j)}.
\end{aligned}
$$
Thus, choosing $\M = (m_{ij})_{i,j = 1}^n \in \mathbb{R}^{n \times n}$ with
$$m_{ij} = \begin{cases}
1+n \lambda_i^2 - 2 \lambda_i,      & i = j, \\
n\lambda_i \lambda_j - 2 \lambda_j,  & i \neq j,
\end{cases}$$
we get that $\bfrak(\x,\rho) = \QO(\x, \M)$.
\end{example}

\begin{example}[Variance]
The variance of the vector $\x$ can be obtained as a quadratic ordered measure by choosing:
$$\M = \frac{1}{n} \cdot \begin{pmatrix}
1-\frac{1}{n}  & -\frac{1}{n}  & \ldots       & -\frac{1}{n} \\ 
-\frac{1}{n}   & 1-\frac{1}{n} & \ddots       & \vdots       \\ 
\vdots         & \ddots        & \ddots       & -\frac{1}{n} \\ 
-\frac{1}{n}   & \ldots        & -\frac{1}{n} & 1-\frac{1}{n}
\end{pmatrix}$$
then $\QO(\x , \M)$ represents the variance of the values, i.e.,
$$\displaystyle{\frac{1}{n} \sum_{i=1}^n \big(x_i - \frac{1}{n} \sum_{j=1}^n x_j \big)^2}.$$
\end{example}

\begin{example}[$(r,s)$-trimmed variance]
The proposed extension allows to express general measures based on deviations, as shown in Table \ref{table:dispersion_lambdas}. This is also the case of trimmed or winsorized variances (see, e.g., \cite{Welsh1990, Wilcox2011}) that can easily be tackled within this framework. For instace, the  $(r,s)$-trimmed variance, whose expression is proportional to
\begin{equation}
    \label{eq:trim-variance}
    \sum_{i=r+1}^{n-s} \big( x_{(i)} - \frac{1}{n - r - s} \sum_{i=r+1}^{n-s} x_{(i)} \big)^2 , 
\end{equation}

can be computed by means of
$$m_{ij} = \begin{cases}
1-\frac{1}{n-r-s}, & r+1 \leq i = j \leq n-s, \\
-\frac{1}{n-r-s},  & r+1 \leq i \neq j \leq n-s, \\
0,                 & otherwise.
\end{cases}$$
\end{example}

\begin{example}[$(r,s)$-winsorized variance]
The $(r,s)$-winsorized variance is proportional to,
\begin{equation}
\hspace{-3cm}  r \big(x_{(r+1)} - W_{(r,s)}(\x)\big)^2 + \sum_{i=r+1}^{n-s} \big(x_{(i)} - W_{(r,s)}(\x) \big)^2 + s \big(x_{(n-s)} - W_{(r,s)}(\x) \big)^2 \hspace{-3cm} \label{eq:win-variance}
\end{equation}
where $W_{(r,s)}(\x)$ is the $(r,s)$-winsorized mean in Table \ref{table:location_lambdas}, by means of 
$$m_{ij} = \begin{cases}
1+r-\frac{(r+1)^2}{n}, & i = j = r+1, \\
1+s-\frac{(s+1)^2}{n}, & i = j = n-s, \\
1-\frac{1}{n},          & r+1 < i=j < n-s, \\
-\frac{(rs+r+s+1)}{n}, & i=r+1 \text{ and } j=n-1 \text{ (and vicerversa)}, \\
-\frac{(r+1)}{n},      & i=r+1 \text{ and } r+1<j<n-s \text{ (and vicerversa)}, \\
-\frac{(s+1)}{n},      & i=n-s \text{ and } r+1<j<n-s \text{ (and vicerversa)}, \\
-\frac{1}{n},         & r+1 < i \neq j < n-s, \\
0,                     & otherwise.
\end{cases}$$
\end{example}

The wide range of  measures that can be derived from the new family of statistics that we introduced would be only useful if they can be efficiently calculated and capable of being inserted into complex decision-making problems where the values of the vector $\x$ are part of the decision to make.

As already analyzed for $\tL$-measures and nested measures in the previous sections. In what follows, we provide unified suitable mathematical optimization models  to compute $\QO_{n}(\x, \M)$ for any matrix $\M$.
In order to provide an integer mathematical programming formulation to compute this family of quadratic ordered measures, let us resort to the formulation \eqref{form:L_1} introduced in Section \ref{sec:3.1}, and consider the binary assignment variables $z_{ik} \in \{0,1\}, \forall i, k \in N,$ and the set of constraints:
$$
\mathcal{O}(\x) = \Big\{z \in \{0,1\}^{n \times n}: \sum_{i\in N} z_{ik}=1,  \forall k \in N,  \sum_{k\in N} z_{ik}=1, \forall i \in N, \sum_{i\in N} z_{ik} x_{i} \leq \sum_{i\in N} z_{i,k+1} x_{i}, \forall k \in N (k<n)\Big\},
$$
that, as already mentioned determines the set of feasible assignments for the different values in $\x$ to its sorted position, i.e., any $\bar z \in \mathcal{O}(\x)$ with $\bar z_{ik}=1$ indicates that $x_i$ is sorted in $k$-th position in $\{x_1, \ldots, x_n\}$. Note that if all the components of $\x$ are different, this set consists of a single vector, otherwise, in case of ties, the possible sorting of $\x$ is not unique, and then, multiple vectors may be included in the set $\mathcal{O}(\x)$.

The computation of the $\QO$-measure for the vector $\x$ can be addressed via the following integer quadratic formulation:
\begin{align}
\min_{z \in \mathcal{O}(x)} & \  \sum_{i,k \in N} \sum_{j,\ell \in N} m_{k \ell} x_i x_j z_{ik} z_{j \ell}, & \label{cons:Q_0}
\end{align}

Note that this nonlinear problem can be linearized via standard techniques, as McCormick envelopes~\cite{Mccormick1976}. To this end, let us consider binary variables denoting the product of the variables, i.e.,  $y_{ijkl} = z_{ik} z_{j \ell} \in \{0,1\}$, for all $i, j, k, \ell \in N$, then a linear formulation to compute (\ref{eq:QO_M}) can be introduced as follows:
\begin{subequations}
\label{form:Q_linearized}
\begin{align}
\min_{z \in \mathcal{O}(\x)} & \  \sum_{i,k =1}^n \sum_{j,\ell =1}^n m_{k \ell} x_i x_j y_{ijk\ell}, & \label{cons:Qlin_0}  \\ 
\mbox{s.t. } & y_{ijk\ell} \leq z_{ik},    & \forall i, k \in N,       \label{cons:Qlin_1} \\
& y_{ijk\ell} \leq z_{j \ell},             & \forall j,\ell \in N,     \label{cons:Qlin_2} \\
& y_{ijk\ell} \geq z_{ik} + z_{j \ell} -1, & \forall i,j,k,\ell \in N, \label{cons:Qlin_3} \\
& y_{ijk\ell}  \geq 0,                     & \forall i,j,k,\ell \in N. \label{cons:Qlin_4} 
\end{align}
\end{subequations}

Observe that, although the measures provided in Section \ref{sec:3.2} can also result in quadratic models (based on the expression of $\rho$), their objective function would lead to a nonconvex quadratic objective function (implied by the product of continuous variables). In the best-case scenario, only few of such products are present, preserving the model’s competitiveness. In contrast, the framework proposed in this section offers a greater flexibility and simplicity but requires binary variables with two indices, which may increase computational complexity and limit scalability. Given the problem's complexity, the choice between these frameworks depends on the specific application.

Note that one can easily extend this type of  quadratic measures to include the $\tL$-measures described above, by considering non-squared matrices $(\lambda|\M) \in \mathbb{R}^{n \times (n+1)}$, in whose case, the $\QO$ problem described above results in the following measure:
$$
\QO(\x, (\lambda|\M)) = \sum_{k=1}^n \lambda_k x_{(k)} + \sum_{i,j=1}^n m_{ij} x_{(i)} x_{(j)} 
$$
for $\x \in \R^n$.

Thus, taking $\M=\mathbf{0}$, $\QO(\x, (\lambda|\M)) = \ell(\x,\lambda)$, whereas if $\lambda=0$, $\QO(\x, (\lambda|\M)) = \QO(\x, \M)$. This general measure can be formulated as the following optimization problem by means of the $z$-variables described above:
\begin{align}
\min_{z \in \mathcal{O}(\x)} & \  \sum_{i,k=1}^n \lambda_k x_i z_{ik} + \sum_{i,k =1}^n \sum_{j,\ell =1}^n m_{k \ell} x_i x_j z_{ik} z_{j \ell}. & \label{eq:opt_pb2}
\end{align}

\section{Computational experiments}
\label{sec:5}

This section presents the results obtained from our computational experiments, conducted to empirically compare the proposed models. Our formulations were implemented in Python (version 3.10) using the Gurobi Optimizer (version 9.5) as the MIP solver. The experiments were conducted on a MacPro server equipped with a 2.7 GHz Intel Xeon W processor (24 cores) and 192 GB RAM, utilizing 8 threads per run. A time limit of 3600 seconds was set for computations.
For evaluating different metric computations, including the $\tL$-measures, nested $\tL$-measures, and $\QO$-measures, we generated 20 data vectors of sizes 10, 20, 30, 40, 50, 100, 500 and 1000. Each element in the vectors $(x_i)_{i=1}^n$ followed a uniform distribution $U([0,100])$ and was rounded to three decimal digits.

Two primary metrics were used to assess model efficiency. First, \texttt{Time} represents the average time (in seconds) required to compute the best solution for an instance over 20 trials. Second,  \texttt{Accuracy} shows the number of instances in which the value obtained by our models closely matches the reference values computed using Python’s state-of-the-art functions from the \texttt{numpy}, \texttt{statictics}, and \texttt{scipy} packages. Specifically, given $\texttt{f}_{\text{python}}$ as the value computed using the Python package and $\texttt{f}_{\text{m}}$ as the solution obtained by our models, \texttt{Accuracy} shows the number of instances for which
$(\texttt{f}_{\text{python}} - \texttt{f}_{\text{m}}) \leq 10^{-5}$. In the worst case, our solution will be suboptimal with respect to the one computed by python, i.e, $\texttt{f}_{\text{m}} \leq \texttt{f}_{\text{python}}$.  

Table \ref{tab:results-Lmeasures} shows the results for some $\tL$-measures, which are: the \texttt{mean}, $(n/10,n/10)$-trimmed mean (\texttt{trim-mean}), $(n/10,n/10)$-winsorized mean (\texttt{win-mean}), \texttt{range}, $(n/10,n/10)$-mid mean (\texttt{mid-mean}), \texttt{median}, first and third quartiles (\texttt{1quartile} and \texttt{3quartile}), the interquartile range (\texttt{iqr}), mean absolute deviations with respect to the median (\texttt{madev-median}) and the Gini differences (\texttt{gini-difs}), which expressions can be found in Tables \ref{table:location_lambdas} and \ref{table:dispersion_lambdas}. The tested models are $\ell_1$, which is the one derived from formulation \eqref{form:L_1}, $\ell_2$ obtained from formulation \eqref{form:L_2}, and $\ell_3$ obtained from formulation \eqref{form:L_3}. The results show that while $\ell_1$ can handle small-scale problems effectively, it struggles with larger datasets, reaching the time limit of 3600 seconds. Conversely, $\ell_2$ and $\ell_3$ demonstrate significantly faster computational times, maintaining high accuracy across all instances.

\begin{table}[h!]
\centering
\tiny
\renewcommand{\arraystretch}{0.75}
\begin{tabular}{@{}cc|rrrrrrrr|rrrrrrrr@{}}
\toprule
\multicolumn{1}{c}{\multirow{2}{*}{Measure}} &
  \multicolumn{1}{c}{\multirow{2}{*}{Model}} &
  \multicolumn{8}{c|}{Time} &
  \multicolumn{8}{c}{Accuracy} \\ \cmidrule(l){3-18} 
\multicolumn{1}{c}{} &
  \multicolumn{1}{c}{} &
  10 & 20 & 30 & 40 & 50 & 100 & 500 & 1000  & 10 & 20 & 30 & 40 & 50 & 100 & 500 & 1000  \\ \midrule
\multirow{3}{*}{\texttt{mean}}         & $\ell_1$ & 0.0012 & 0.0055 & 0.0173 & 0.0437 & 0.1533 &  52.5193 & 3600.0000 & 3600.0000  & 20 & 20 & 20 & 20 & 20 & 20 & 20 & 20   \\
                                       & $\ell_2$ & 0.0001 & 0.0001 & 0.0002 & 0.0002 & 0.0002 &   0.0007 &    0.0051 &    0.0081  20 & 20 & 20 & 20 & 20 & 20 & 20 & 20  &20 \\
                                       & $\ell_3$ & 0.0001 & 0.0001 & 0.0001 & 0.0001 & 0.0001 &   0.0002 &    0.0029 &    0.0005  20 & 20 & 20 & 20 & 20 & 20 & 20 & 20 &20  \\* \midrule
\multirow{3}{*}{\texttt{trim-mean}}    & $\ell_1$ & 0.0011 & 0.0055 & 0.0173 & 0.0435 & 0.1236 &  83.1352 & 3600.0000 & 3600.0000 &   20 & 20 & 20 & 20 & 20 & 20 &  0 &  0   \\
                                       & $\ell_2$ & 0.0001 & 0.0001 & 0.0002 & 0.0002 & 0.0003 &   0.0007 &    0.0083 &    0.0118  20 & 20 & 20 & 20 & 20 & 20 & 20 & 20 &20  \\
                                       & $\ell_3$ & 0.0001 & 0.0001 & 0.0002 & 0.0002 & 0.0003 &   0.0009 &    0.0103 &    0.0106  20 & 20 & 20 & 20 & 20 & 20 & 20 & 20 &20  \\* \midrule
\multirow{3}{*}{\texttt{win-mean}}     & $\ell_1$ & 0.0012 & 0.0055 & 0.0172 & 0.0437 & 0.1450 &  55.1646 & 3600.0000 & 3600.0000 &   20 & 20 & 20 & 20 & 20 & 20 &  0 &  0  \\
                                       & $\ell_2$ & 0.0002 & 0.0002 & 0.0003 & 0.0003 & 0.0004 &   0.0014 &    0.0097 &    0.0195 & 20 & 20 & 20 & 20 & 20 & 20 & 20 & 20   \\
                                       & $\ell_3$ & 0.0002 & 0.0002 & 0.0003 & 0.0004 & 0.0005 &   0.0015 &    0.0155 &    0.0281 & 20 & 20 & 20 & 20 & 20 & 20 & 20 & 20   \\* \midrule
\multirow{3}{*}{\texttt{range}}        & $\ell_1$ & 0.0011 & 0.0055 & 0.0173 & 0.0436 & 0.1528 &  52.6488 & 3600.0000 & 3600.0000   & 20 & 20 & 20 & 20 & 20 & 20 &  0 &  0   \\
                                       & $\ell_2$ & 0.0002 & 0.0002 & 0.0002 & 0.0003 & 0.0003 &   0.0010 &    0.0064 &    0.0090 & 20 & 20 & 20 & 20 & 20 & 20 & 20 & 20   \\
                                       & $\ell_3$ & 0.0002 & 0.0002 & 0.0003 & 0.0003 & 0.0004 &   0.0014 &    0.0094 &    0.0162   & 20 & 20 & 20 & 20 & 20 & 20 & 20 & 20   \\* \midrule
\multirow{3}{*}{\texttt{mid-mean}}     & $\ell_1$ & 0.0011 & 0.0055 & 0.0174 & 0.0437 & 0.1464 &  55.8178 & 3600.0000 & 3600.0000   & 20 & 20 & 20 & 20 & 20 & 20 &  0 &  0   \\
                                       & $\ell_2$ & 0.0002 & 0.0002 & 0.0002 & 0.0003 & 0.0003 &   0.0007 &    0.0064 &    0.0154   & 20 & 20 & 20 & 20 & 20 & 20 & 20 & 20   \\
                                       & $\ell_3$ & 0.0002 & 0.0002 & 0.0002 & 0.0003 & 0.0004 &   0.0010 &    0.0088 &    0.0122   & 20 & 20 & 20 & 20 & 20 & 20 & 20 & 20   \\* \midrule
\multirow{3}{*}{\texttt{median}}       & $\ell_1$ & 0.0011 & 0.0054 & 0.0175 & 0.0434 & 0.1458 &   3.2773 & 3600.0000 & 3600.0000 &   20 & 20 & 20 & 20 & 20 & 20 &  0 &  0  \\
                                       & $\ell_2$ & 0.0002 & 0.0001 & 0.0002 & 0.0003 & 0.0003 &   0.0012 &    0.0088 &    0.0148   & 20 & 20 & 20 & 20 & 20 & 20 & 20 & 20   \\
                                       & $\ell_3$ & 0.0001 & 0.0002 & 0.0002 & 0.0003 & 0.0003 &   0.0013 &    0.0142 &    0.0220   & 20 & 20 & 20 & 20 & 20 & 20 & 20 & 20   \\* \midrule
\multirow{3}{*}{\texttt{1quartile}}    & $\ell_1$ & 0.0011 & 0.0055 & 0.0172 & 0.0438 & 0.1665 & 107.7195 & 3600.0000 & 3600.0000 &   20 & 20 & 20 & 20 & 20 & 20 &  0 &  0   \\
                                       & $\ell_2$ & 0.0002 & 0.0002 & 0.0002 & 0.0002 & 0.0003 &   0.0009 &    0.0086 &    0.0141   & 20 & 20 & 20 & 20 & 20 & 20 & 20 & 20   \\
                                       & $\ell_3$ & 0.0001 & 0.0001 & 0.0002 & 0.0002 & 0.0003 &   0.0010 &    0.0119 &    0.0157   & 20 & 20 & 20 & 20 & 20 & 20 & 20 & 20   \\* \midrule
\multirow{3}{*}{\texttt{3quartile}}    & $\ell_1$ & 0.0012 & 0.0055 & 0.0172 & 0.0436 & 0.1350 &  51.9655 & 3600.0000 & 3600.0000 &   20 & 20 & 20 & 20 & 20 & 20 &  0 &  0   \\
                                       & $\ell_2$ & 0.0002 & 0.0001 & 0.0002 & 0.0002 & 0.0003 &   0.0009 &    0.0070 &    0.0106   & 20 & 20 & 20 & 20 & 20 & 20 & 20 & 20   \\
                                       & $\ell_3$ & 0.0001 & 0.0002 & 0.0002 & 0.0003 & 0.0003 &   0.0012 &    0.0122 &    0.0158   & 20 & 20 & 20 & 20 & 20 & 20 & 20 & 20   \\* \midrule
\multirow{3}{*}{\texttt{iqr}}          & $\ell_1$ & 0.0012 & 0.0055 & 0.0170 & 0.0433 & 0.1405 &  58.9756 & 3600.0000 & 3600.0000 &   20 & 20 & 20 & 20 & 20 & 20 &  0 &  0   \\
                                       & $\ell_2$ & 0.0002 & 0.0002 & 0.0003 & 0.0004 & 0.0005 &   0.0013 &    0.0129 &    0.0325   & 20 & 20 & 20 & 20 & 20 & 20 & 20 & 20   \\
                                       & $\ell_3$ & 0.0002 & 0.0002 & 0.0003 & 0.0005 & 0.0006 &   0.0018 &    0.0178 &    0.0426   & 20 & 20 & 20 & 20 & 20 & 20 & 20 & 20   \\* \midrule
\multirow{3}{*}{\texttt{madev-median}} & $\ell_1$ & 0.0011 & 0.0054 & 0.0172 & 0.0435 & 0.1446 &  52.9097 & 3600.0000 & 3600.0000 &   20 & 20 & 20 & 20 & 20 & 20 &  0 &  0   \\
                                       & $\ell_2$ & 0.0001 & 0.0002 & 0.0002 & 0.0003 & 0.0003 &   0.0008 &    0.0085 &    0.0143   & 20 & 20 & 20 & 20 & 20 & 20 & 20 & 20   \\
                                       & $\ell_3$ & 0.0001 & 0.0002 & 0.0002 & 0.0003 & 0.0003 &   0.0009 &    0.0112 &    0.0151   & 20 & 20 & 20 & 20 & 20 & 20 & 20 & 20   \\* \midrule
\multirow{3}{*}{\texttt{gini-difs}}    & $\ell_1$ & 0.0011 & 0.0054 & 0.0168 & 0.0428 & 0.1403 &  52.7795 & 3600.0000 & 3600.0000 &   20 & 20 & 20 & 20 & 20 & 20 &  0 &  0   \\
                                       & $\ell_2$ & 0.0004 & 0.0012 & 0.0028 & 0.0060 & 0.0121 &   0.0249 &    1.7334 &   16.9960   & 20 & 20 & 20 & 20 & 20 & 20 & 20 & 20   \\
                                       & $\ell_3$ & 0.0004 & 0.0015 & 0.0039 & 0.0086 & 0.0176 &   0.0404 &    2.1569 &   20.1885   & 20 & 20 & 20 & 20 & 20 & 20 & 20 & 20   \\* \bottomrule
\end{tabular}
\caption{Results obtained evaluating $\tL$ measures.}
\label{tab:results-Lmeasures}
\end{table}

Table \ref{tab:results-nestedLmeasures} presents the results for several nested $\tL$-measures, including the \texttt{variance} (mean squared deviations from the mean), mean squared deviations from the median, maximum, and minimum (\texttt{msdev-median}, \texttt{msdev-max}, and \texttt{msdev-min}), as well as their corresponding mean absolute deviations (\texttt{madev-median}, \texttt{madev-max}, and \texttt{madev-min}). These measures, previously introduced in Table \ref{table:dispersion_lambdas}, provide insight into the dispersion characteristics of the data. Additionally, the table includes the well-known \texttt{skewness} and \texttt{kurtosis} coefficients, where the mean serves as the nested location measure.
The results compare two models: $\bfrak_1$, which is derived from formulation \eqref{form:N_1}, and $\bfrak_2$, obtained from formulation \eqref{form:N_2}. The computational time and accuracy for each measure are reported across varying sample sizes.
Both models, $\bfrak_1$ and $\bfrak_2$, exhibit similar performance increasing computational time as the sample size grows, except for large instances sizes where $\bfrak_1$ is outperformed by $\bfrak_2$ in terms of time, especially for the \texttt{madev-max} and \texttt{skewness} criteria. 
Also, both $\bfrak_1$ and $\bfrak_2$ demonstrate stable accuracy across all nested $\tL$-measures, maintaining a consistent score of 20 in most cases. This suggests that both models provide reliable estimates across different dispersion and shape measures, except for the \texttt{msdev-min} measure where $\bfrak_1$ is clearly outperformed by $\bfrak_2$ for large instances.

\begin{table}[h!]
\centering
\tiny
\renewcommand{\arraystretch}{0.75}
\begin{tabular}{@{}cc|rrrrrrrr|rrrrrrrr@{}}
\toprule
\multicolumn{1}{c}{\multirow{2}{*}{Measure}} &
  \multicolumn{1}{c}{\multirow{2}{*}{Model}} &
  \multicolumn{8}{c|}{Time} &
  \multicolumn{8}{c}{Accuracy} \\ \cmidrule(l){3-18} 
\multicolumn{1}{c}{} &
  \multicolumn{1}{c}{} &
  10 & 20 & 30 & 40 & 50 & 100 & 500 & 1000  & 10 & 20 & 30 & 40 & 50 & 100 & 500 & 1000  \\ \midrule
\multirow{2}{*}{\texttt{variance}}     & $\bfrak_1$ & 0.0003 & 0.0005 & 0.0009 & 0.0014 & 0.0021 & 0.0098 &  0.3784 &  1.8353   & 20 & 20 & 20 & 20 & 20 &20 & 20 & 20   \\
                                       & $\bfrak_2$ & 0.0002 & 0.0004 & 0.0009 & 0.0016 & 0.0024 & 0.0107 &  0.3871 &  2.1006   & 20 & 20 & 20 & 20 & 20 &20 & 20 & 20   \\* \midrule
\multirow{2}{*}{\texttt{msdev-median}} & $\bfrak_1$ & 0.0004 & 0.0008 & 0.0013 & 0.0023 & 0.0034 & 0.0131 &  0.3337 &  1.7240   & 20 & 20 & 20 & 20 & 20 &20 & 20 & 20   \\
                                       & $\bfrak_2$ & 0.0004 & 0.0009 & 0.0016 & 0.0029 & 0.0045 & 0.0166 &  0.3927 &  2.0073   & 20 & 20 & 20 & 20 & 20 &20 & 20 & 20   \\* \midrule
\multirow{2}{*}{\texttt{msdev-max}}    & $\bfrak_1$ & 0.0003 & 0.0006 & 0.0010 & 0.0017 & 0.0025 & 0.0107 &  0.3611 &  1.6875   & 20 & 20 & 20 & 20 & 20 &20 & 20 & 20   \\
                                       & $\bfrak_2$ & 0.0003 & 0.0007 & 0.0013 & 0.0023 & 0.0034 & 0.0132 &  0.4167 &  1.9957   & 20 & 20 & 20 & 20 & 20 &20 & 20 & 20   \\* \midrule
\multirow{2}{*}{\texttt{msdev-min}}    & $\bfrak_1$ & 0.0004 & 0.0007 & 0.0013 & 0.0023 & 0.0034 & 0.0136 &  0.3471 &  1.8349   & 20 & 20 & 20 & 20 & 20 &19 & 14 & 10   \\
                                       & $\bfrak_2$ & 0.0003 & 0.0007 & 0.0012 & 0.0021 & 0.0034 & 0.0130 &  0.4140 &  1.9671   & 20 & 20 & 20 & 20 & 20 &20 & 20 & 20   \\* \midrule
\multirow{2}{*}{\texttt{madev-mean}}   & $\bfrak_1$ & 0.0006 & 0.0014 & 0.0029 & 0.0054 & 0.0129 & 0.1243 &  0.5407 &  2.6555   & 20 & 20 & 20 & 20 & 20 &20 & 20 & 20   \\
                                       & $\bfrak_2$ & 0.0003 & 0.0006 & 0.0011 & 0.0020 & 0.0032 & 0.0152 &  0.5710 &  2.7334   & 20 & 20 & 20 & 20 & 20 &20 & 20 & 20   \\* \midrule
\multirow{2}{*}{\texttt{madev-median}} & $\bfrak_1$ & 0.0011 & 0.0019 & 0.0031 & 0.0047 & 0.0066 & 0.0270 &  0.6400 &  3.0635   & 20 & 20 & 20 & 20 & 20 &20 & 20 & 20   \\
                                       & $\bfrak_2$ & 0.0012 & 0.0021 & 0.0035 & 0.0054 & 0.0075 & 0.0279 &  0.8018 &  3.5806   & 20 & 20 & 20 & 20 & 20 &20 & 20 & 20   \\* \midrule
\multirow{2}{*}{\texttt{madev-max}}    & $\bfrak_1$ & 0.0008 & 0.0012 & 0.0016 & 0.0025 & 0.0036 & 0.0167 &  6.4872 & 51.1322   & 20 & 20 & 20 & 20 & 20 &20 & 20 & 20   \\
                                       & $\bfrak_2$ & 0.0008 & 0.0017 & 0.0031 & 0.0067 & 0.0104 & 0.0308 &  0.7752 &  3.9238   & 20 & 20 & 20 & 20 & 20 &20 & 20 & 20   \\* \midrule
\multirow{2}{*}{\texttt{madev-min}}    & $\bfrak_1$ & 0.0013 & 0.0034 & 0.0074 & 0.0093 & 0.0133 & 0.0439 &  0.7684 &  3.2349   & 20 & 20 & 20 & 20 & 20 &20 & 20 & 20   \\
                                       & $\bfrak_2$ & 0.0005 & 0.0009 & 0.0015 & 0.0028 & 0.0041 & 0.0180 &  0.6795 &  3.2782   & 20 & 20 & 20 & 20 & 20 &20 & 20 & 20   \\* \midrule
\multirow{2}{*}{\texttt{skewness}}     & $\bfrak_1$ & 0.0018 & 0.0030 & 0.0054 & 0.0099 & 0.0204 & 0.1950 &  0.9368 & 52.8096   & 20 & 20 & 20 & 20 & 20 &20 & 20 & 20   \\
                                       & $\bfrak_2$ & 0.0014 & 0.0019 & 0.0031 & 0.0050 & 0.0073 & 0.0286 & 10.4145 &  5.0563   & 20 & 20 & 20 & 20 & 20 &20 & 20 & 20  \\* \midrule
\multirow{2}{*}{\texttt{kurtosis}}     & $\bfrak_1$ & 0.0016 & 0.0029 & 0.0053 & 0.0095 & 0.0205 & 0.2081 &  0.9167 &  4.6692   & 20 & 20 & 20 & 20 & 20 &20 & 20 & 20   \\
                                       & $\bfrak_2$ & 0.0012 & 0.0019 & 0.0031 & 0.0051 & 0.0073 & 0.0302 &  1.0423 &  5.0956   & 20 & 20 & 20 & 20 & 20 &20 & 20 & 20   \\* \bottomrule
\end{tabular}
\caption{Results obtained evaluating nested measures.}
\label{tab:results-nestedLmeasures}
\end{table}

Finally, Table \ref{tab:results-Qmeasures} summarizes the results for several $\QO$-measures, including mean squared deviations from the mean, median, maximum, and minimum (\texttt{variance}, \texttt{msdev-median}, \texttt{msdev-max}, and \texttt{msdev-min}, respectively). Additionally, the table includes the $(n/10, n/10)$-trimmed variance (\texttt{trim-variance}) and the $(n/10, n/10)$-winsorized variance (\texttt{win-variance}), which were introduced in Section \ref{sec:4}. The only model tested for this set of measures is derived from formulation \eqref{form:Q_linearized}, denoted as $\QO$.
For the $\QO$-measures the computational time increases significantly as the sample size grows.  The model requires considerably more time as the sample size grows, indicating a higher computational cost compared to the $\tL$-measure and nested $\tL$-measure models. 
In contrast, the accuracy of the model also remains stable at 20 across all tested measures, indicating that despite higher computational costs, the model maintains strong estimation consistency.

\begin{table}[h!]
\centering
\footnotesize
\renewcommand{\arraystretch}{0.75}
\begin{tabular}{@{}cc|rrrrr|rrrrr@{}}
\toprule
\multicolumn{1}{c}{\multirow{2}{*}{Measure}} &
  \multicolumn{1}{c}{\multirow{2}{*}{Model}} &
  \multicolumn{5}{c|}{Time} &
  \multicolumn{5}{c}{Accuracy} \\ \cmidrule(l){3-12} 
\multicolumn{1}{c}{} &
  \multicolumn{1}{c}{} &
  10 & 20 & 30 & 40 & 50 & 10 & 20 & 30 & 40 & 50 \\ \midrule
{\texttt{variance}}      & $\QO$ &  9.1381 & 13.6615 & 128.1267 & 902.3868 & 3292.905 & 20 & 20 & 20 & 20 & 20 \\* \midrule
{\texttt{msdev-median}}  & $\QO$ &  2.8683 & 13.9078 & 125.0910 & 793.5474 & 2187.096 & 20 & 20 & 20 & 20 & 20 \\* \midrule
{\texttt{msdev-max}}     & $\QO$ &  2.1225 & 12.4740 & 118.8592 & 746.7027 & 2225.038 & 20 & 20 & 20 & 20 & 20 \\* \midrule
{\texttt{msdev-min}}     & $\QO$ &  1.7471 & 12.2372 & 118.0600 & 750.9281 & 2199.251 & 20 & 20 & 20 & 20 & 20 \\* \midrule
{\texttt{trim-variance}} & $\QO$ & 11.0785 & 13.0243 & 119.2869 & 844.5739 & 2952.002 & 20 & 20 & 20 & 20 & 20 \\* \midrule
{\texttt{win-variance}}  & $\QO$ & 10.5823 & 12.8667 & 120.0339 & 855.3233 & 3224.855 & 20 & 20 & 20 & 20 & 20 \\* \bottomrule
\end{tabular}
\caption{Results obtained evaluating $\QO$ measures.}
\label{tab:results-Qmeasures}
\end{table}

\hfill 

\begin{rem}
    Note that, although it is not our goal,  methodology allows evaluating the measures using the \textit{a priori} information about the sorted values of $\x$ by fixing the values of the variables in the optimization models. A thorough analysis of the different systems of constraints results in a worst-case complexity of $O(n^2)$ for the $\tL$-measures and the nested measures, and $O(n^4)$ for the quadratic $\QO$ measures.
\end{rem}

\section{Application to optimization problems}
\label{sec:6}

As already mentioned, the main advantage of the optimization-based methodologies that we propose to compute the metrics analyzed in this paper is that one can embed them into decision problems that requires optimizing considering that these measures depend on variable values that are part of the decision making process. 
In this section we focus on applying these general measures to values $\x$ that are  decision variables of feasible sets of different problems. To illustrate our claim, we show the use of these measures on three challenging  optimization problems  where this aggregation is particularly useful: scenario analysis in linear programming, traveling salesmam problem and weighted multicover set problem.

\subsection{Scenario Analysis in Linear Programming}\label{sec:polyhedron}

Given a set of $n$ cost scenarios $c = (c_1, \ldots, c_n) \in \R^{d\times n}$ in a standard linear programming problem, the goal of scenario analysis is to derive solutions of the linear programming problem under different metrics to highlight qualitative properties of these solutions. Then, we are given a polyhedron $P = \{z \in \R_+^d: Az \leq b\}$, and $\x = C z \in \R^n$ for $z\in P$.

We wish to solve the problem $\ell(\x,\blambda)$, $\mathfrak{b}(\x,\blambda)$, and $\QO(\x,M)$ for different values of the $\blambda$ and $M$.

For illustrative purposes, we solve various problems for a planar ($d=2$) polyhedron, $P$, with vertices  $V =\{(2, 3), (5, 1), (7, 4),  (4, 8),  (1, 6)\}$, which is drawn in Figure \ref{fig:polyhedron} and cost matrix:
$$
C = \begin{pmatrix}
      3 &  4\\
      2 &  0\\
-3 &  -2\\
-2 & -6\\
4 &-10
\end{pmatrix}
$$

We run our models for the $\ell$, $\mathfrak{b}$, and $\QO$ models for different $\blambda$-vectors and matrices, to detect an optimal point in the polyhedron minimizing the measure aggregating the cost scenarios that we analyze in this paper. We use the $\blambda$ vectors identified with the mean, the maximum, and their nested square deviation versions with respect to those ordered measures. 

The solutions are plotted in Figure \ref{fig:polyhedron} (right), where we observe that the solutions obtained with the different summarizing criteria differ from each other.

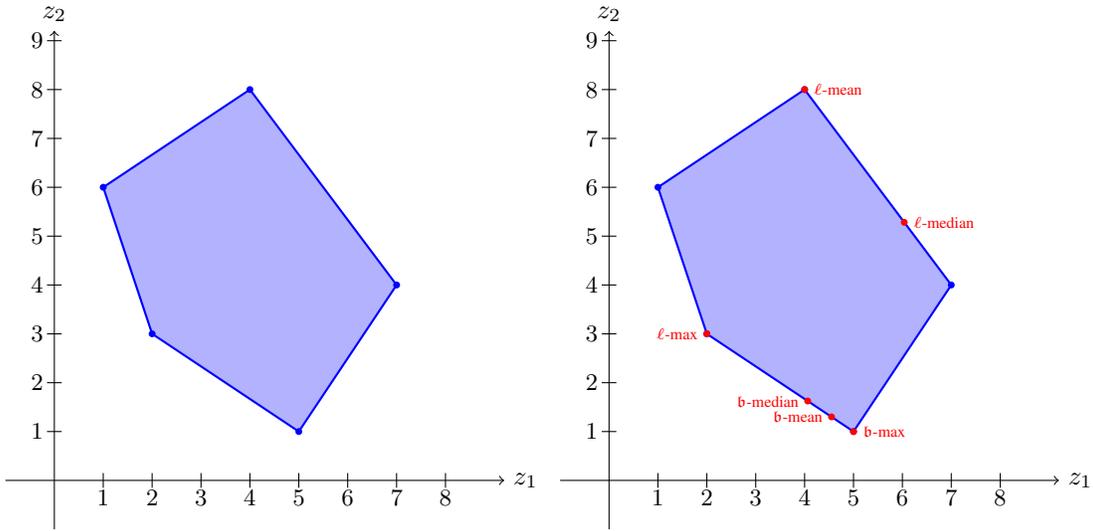
\begin{figure}[h]
\begin{center}
\begin{tikzpicture}[scale=0.65]
    \draw[thin,->] (-1,0) -- (9.2,0) node[right] {$z_1$};
    \draw[thin,->] (0,-1) -- (0,9.2) node[above] {$z_2$};

        \coordinate (A) at (2,3);
        \coordinate (B) at (5,1);
        \coordinate (C) at (7,4);
        \coordinate (D) at (4,8);
        \coordinate (E) at (1,6);

        \fill[blue!30, opacity=0.5] (A) -- (B) -- (C) -- (D) -- (E) -- cycle;

        \draw[thick, blue] (A) -- (B) -- (C) -- (D) -- (E) -- cycle;

        \foreach \point in {A, B, C, D, E}
            \fill[blue] (\point) circle (2pt);

    \node at (-0.35,1) {\small $1$};
    \node at (-0.35,2) {\small $2$};
    \node at (-0.35,3) {\small $3$};
    \node at (-0.35,4) {\small $4$};
    \node at (-0.35,5) {\small $5$};
    \node at (-0.35,6) {\small $6$};
    \node at (-0.35,7) {\small $7$};
    \node at (-0.35,8) {\small $8$};
    \node at (-0.35,9) {\small $9$};
    \draw[thin] (0.15,1)--(-0.15,1);
    \draw[thin] (0.15,2)--(-0.15,2);
    \draw[thin] (0.15,3)--(-0.15,3);
    \draw[thin] (0.15,4)--(-0.15,4);
    \draw[thin] (0.15,5)--(-0.15,5);
    \draw[thin] (0.15,6)--(-0.15,6);
    \draw[thin] (0.15,7)--(-0.15,7);
    \draw[thin] (0.15,8)--(-0.15,8);
    \draw[thin] (0.15,9)--(-0.15,9);

    \draw[thin] (1,0.15)--(1,-0.15);
    \draw[thin] (2,0.15)--(2,-0.15);
    \draw[thin] (3,0.15)--(3,-0.15);
    \draw[thin] (4,0.15)--(4,-0.15);
    \draw[thin] (5,0.15)--(5,-0.15);
    \draw[thin] (6,0.15)--(6,-0.15);
    \draw[thin] (7,0.15)--(7,-0.15);
    \draw[thin] (8,0.15)--(8,-0.15);

    \node at (1,-0.35) {\small $1$};
    \node at (2,-0.35) {\small $2$};
    \node at (3,-0.35) {\small $3$};
    \node at (4,-0.35) {\small $4$};
    \node at (5,-0.35) {\small $5$};
    \node at (6,-0.35) {\small $6$};
    \node at (7,-0.35) {\small $7$};
    \node at (8,-0.35) {\small $8$};


\end{tikzpicture}~\begin{tikzpicture}[scale=0.65]
    \draw[thin,->] (-1,0) -- (9.2,0) node[right] {$z_1$};
    \draw[thin,->] (0,-1) -- (0,9.2) node[above] {$z_2$};

    \coordinate (A) at (2,3);
        \coordinate (B) at (5,1);
        \coordinate (C) at (7,4);
        \coordinate (D) at (4,8);
        \coordinate (E) at (1,6);

        \fill[blue!30, opacity=0.5] (A) -- (B) -- (C) -- (D) -- (E) -- cycle;

        \draw[thick, blue] (A) -- (B) -- (C) -- (D) -- (E) -- cycle;

        \foreach \point in {A, B, C, D, E}
            \fill[blue] (\point) circle (2pt);

    \node at (-0.35,1) {\small $1$};
    \node at (-0.35,2) {\small $2$};
    \node at (-0.35,3) {\small $3$};
    \node at (-0.35,4) {\small $4$};
    \node at (-0.35,5) {\small $5$};
    \node at (-0.35,6) {\small $6$};
    \node at (-0.35,7) {\small $7$};
    \node at (-0.35,8) {\small $8$};
    \node at (-0.35,9) {\small $9$};
    \draw[thin] (0.15,1)--(-0.15,1);
    \draw[thin] (0.15,2)--(-0.15,2);
    \draw[thin] (0.15,3)--(-0.15,3);
    \draw[thin] (0.15,4)--(-0.15,4);
    \draw[thin] (0.15,5)--(-0.15,5);
    \draw[thin] (0.15,6)--(-0.15,6);
    \draw[thin] (0.15,7)--(-0.15,7);
    \draw[thin] (0.15,8)--(-0.15,8);
    \draw[thin] (0.15,9)--(-0.15,9);

    \draw[thin] (1,0.15)--(1,-0.15);
    \draw[thin] (2,0.15)--(2,-0.15);
    \draw[thin] (3,0.15)--(3,-0.15);
    \draw[thin] (4,0.15)--(4,-0.15);
    \draw[thin] (5,0.15)--(5,-0.15);
    \draw[thin] (6,0.15)--(6,-0.15);
    \draw[thin] (7,0.15)--(7,-0.15);
    \draw[thin] (8,0.15)--(8,-0.15);

    \node at (1,-0.35) {\small $1$};
    \node at (2,-0.35) {\small $2$};
    \node at (3,-0.35) {\small $3$};
    \node at (4,-0.35) {\small $4$};
    \node at (5,-0.35) {\small $5$};
    \node at (6,-0.35) {\small $6$};
    \node at (7,-0.35) {\small $7$};
    \node at (8,-0.35) {\small $8$};


    \fill[red] (4., 8.) circle (2pt) node[right] {\tiny $\ell$-mean};
    \fill[red] (2,3) circle (2pt) node[left] {\tiny $\ell$-max};
    \fill[red] (6.03773585, 5.28301887) circle (2pt) node[right] {\tiny $\ell$-median};
    \fill[red] (4.55, 1.3) circle (2pt) node[left] {\tiny $\mathfrak{b}$-mean};
    \fill[red] (5,1) circle (2pt) node[right] {\tiny $\mathfrak{b}$-max};
    \fill[red] (4.0625, 1.625) circle (2pt) node[left] {\tiny $\mathfrak{b}$-median};

\end{tikzpicture}
\end{center}
    \caption{Polyhedron considered in the Example of Section \ref{sec:polyhedron} (left) and solution obtained (right).\label{fig:polyhedron}}
\end{figure}

Note that, although evaluating under these measures for a given set of values $\x$ can be done using sorting approaches, these tools do not allow to determine the solutions of this type of problem that requires the evaluation of infinitely many continuous values.

\subsection{Traveling Salesman Problem}

The traveling salesman problem (TSP) is a well-known combinatorial optimization problem with multiple applications in different field. Given a directed graph $G=(V,A)$, with node set $V$ and arc set $A$, and for each arc $a\in A$, a weight $w_a\geq 0$, the goal is to construct a Hamiltonian cycle visiting the nodes of $G$ exactly once at minimum overall sum of the weights.

Different formulations have been proposed to solve this problem using mathematical optimization models see, e.g., among many others \cite{dfj,gavish1978travelling,miller1960integer}. Most of these formulations use the following family of binary variables to identify the arcs that are part of the solution:

$$
z_a = \begin{cases}
    1 & \text{if $a$ is  in the cycle,}\\
    0 & \text{otherwise}
\end{cases}, \forall a \in A.
$$

Note that since each node appears exactly once as an outgoing node in the cycle, the cost of the outgoing arc from $v\in V$ can be expressed as:

$$
x_v = \sum_{a\in \delta^+(v)} w_a z_a.
$$

Thus, the overall sum of the weights can be written as $\sum_{v\in V} x_v$. Instead of minimizing the overall sum of the weights of the arcs involved in the solution, one can also minimize any of the measures proposed in the previous sections on the set of costs $\{x_v\}_{v\in V}$ described above.

To illustrate the solutions obtained with our methodologies through a simple example, we generate the complete graph drawn in Figure \ref{fig:tsp0}, where the arc costs correspond to the Euclidean norms between the points in the plane representing the nodes.

\begin{figure}[h]
\begin{center}
\begin{tikzpicture}[scale=0.07]
     \draw
        (99.964, 29.443) node[draw, fill=gray!20, minimum size=1pt] (0){1}
        (8.374, 34.1) node[draw, fill=gray!20, minimum size=1pt] (1){2}
        (89.791, 18.66) node[draw, fill=gray!20, minimum size=1pt] (2){3}
        (60.064, 34.381) node[draw, fill=gray!20, minimum size=1pt] (3){4}
        (22.225, 50.676) node[draw, fill=gray!20, minimum size=1pt] (4){5}
        (17.856, 88.361) node[draw, fill=gray!20, minimum size=1pt] (5){6}
        (96.902, 50.93) node[draw, fill=gray!20, minimum size=1pt] (6){7}
        (55.599, 13.246) node[draw, fill=gray!20, minimum size=1pt] (7){8};
      \begin{scope}[->]
        \draw[color=gray, line width=0.5pt] (0) to (1);
        \draw[color=gray, line width=0.5pt] (0) to (2);
        \draw[color=gray, line width=0.5pt] (0) to (3);
        \draw[color=gray, line width=0.5pt] (0) to (4);
        \draw[color=gray, line width=0.5pt] (0) to (5);
        \draw[color=gray, line width=0.5pt] (0) to (6);
        \draw[color=gray, line width=0.5pt] (0) to (7);
        \draw[color=gray, line width=0.5pt] (1) to (0);
        \draw[color=gray, line width=0.5pt] (1) to (2);
        \draw[color=gray, line width=0.5pt] (1) to (3);
        \draw[color=gray, line width=0.5pt] (1) to (4);
        \draw[color=gray, line width=0.5pt] (1) to (5);
        \draw[color=gray, line width=0.5pt] (1) to (6);
        \draw[color=gray, line width=0.5pt] (1) to (7);
        \draw[color=gray, line width=0.5pt] (2) to (0);
        \draw[color=gray, line width=0.5pt] (2) to (1);
        \draw[color=gray, line width=0.5pt] (2) to (3);
        \draw[color=gray, line width=0.5pt] (2) to (4);
        \draw[color=gray, line width=0.5pt] (2) to (5);
        \draw[color=gray, line width=0.5pt] (2) to (6);
        \draw[color=gray, line width=0.5pt] (2) to (7);
        \draw[color=gray, line width=0.5pt] (3) to (0);
        \draw[color=gray, line width=0.5pt] (3) to (1);
        \draw[color=gray, line width=0.5pt] (3) to (2);
        \draw[color=gray, line width=0.5pt] (3) to (4);
        \draw[color=gray, line width=0.5pt] (3) to (5);
        \draw[color=gray, line width=0.5pt] (3) to (6);
        \draw[color=gray, line width=0.5pt] (3) to (7);
        \draw[color=gray, line width=0.5pt] (4) to (0);
        \draw[color=gray, line width=0.5pt] (4) to (1);
        \draw[color=gray, line width=0.5pt] (4) to (2);
        \draw[color=gray, line width=0.5pt] (4) to (3);
        \draw[color=gray, line width=0.5pt] (4) to (5);
        \draw[color=gray, line width=0.5pt] (4) to (6);
        \draw[color=gray, line width=0.5pt] (4) to (7);
        \draw[color=gray, line width=0.5pt] (5) to (0);
        \draw[color=gray, line width=0.5pt] (5) to (1);
        \draw[color=gray, line width=0.5pt] (5) to (2);
        \draw[color=gray, line width=0.5pt] (5) to (3);
        \draw[color=gray, line width=0.5pt] (5) to (4);
        \draw[color=gray, line width=0.5pt] (5) to (6);
        \draw[color=gray, line width=0.5pt] (5) to (7);
        \draw[color=gray, line width=0.5pt] (6) to (0);
        \draw[color=gray, line width=0.5pt] (6) to (1);
        \draw[color=gray, line width=0.5pt] (6) to (2);
        \draw[color=gray, line width=0.5pt] (6) to (3);
        \draw[color=gray, line width=0.5pt] (6) to (4);
        \draw[color=gray, line width=0.5pt] (6) to (5);
        \draw[color=gray, line width=0.5pt] (6) to (7);
        \draw[color=gray, line width=0.5pt] (7) to (0);
        \draw[color=gray, line width=0.5pt] (7) to (1);
        \draw[color=gray, line width=0.5pt] (7) to (2);
        \draw[color=gray, line width=0.5pt] (7) to (3);
        \draw[color=gray, line width=0.5pt] (7) to (4);
        \draw[color=gray, line width=0.5pt] (7) to (5);
        \draw[color=gray, line width=0.5pt] (7) to (6);
      \end{scope}
    \end{tikzpicture}
    \caption{Graph used to illustrate the TSP problem under our measures.\label{fig:tsp0}}
    \end{center}
    \end{figure}
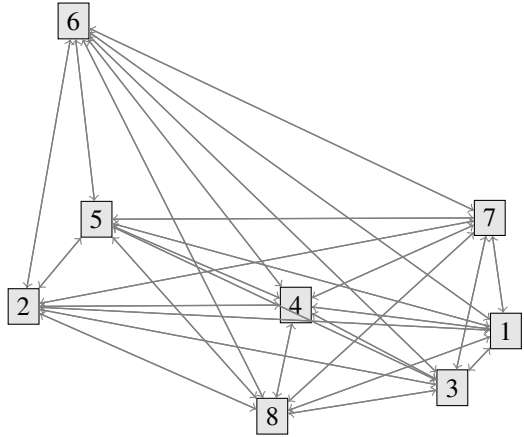

We run our models for the $\ell$, $\mathfrak{b}$, and $\QO$ models for different $\blambda$-vectors and matrices, namely, those $\blambda$ corresponding to the mean, the 25\%-25\%-trimmed mean, and the maximum, their absolute deviations versions for the nested case, and the squared deviations with the $\QO$ approach. In Figures \ref{fig:tsp_solsL}-\ref{fig:tsp_solsQ} we plot the solutions obtained with the different methodologies. In Figure \ref{fig:tsp_solsL} we show the resulting cycles for the $\ell$-measures, in figure \ref{fig:tsp_solsN} the ones obtained for the nested measures, and in Figure \ref{fig:tsp_solsQ} the cycles constructed for the $\QO$ measures based on the squared deviation with respect different $\ell$-measures.

\begin{figure}[h]
\begin{center}
\fbox{\adjustbox{width=0.33\textwidth}{{\centering $\ell$-\texttt{mean}}
  \begin{tikzpicture}[scale=0.1]
      \draw
        (99.964, 29.443) node[draw, fill=gray!20, minimum size=1pt] (0){1}
        (8.374, 34.1) node[draw, fill=gray!20, minimum size=1pt] (1){2}
        (89.791, 18.66) node[draw, fill=gray!20, minimum size=1pt] (2){3}
        (60.064, 34.381) node[draw, fill=gray!20, minimum size=1pt] (3){4}
        (22.225, 50.676) node[draw, fill=gray!20, minimum size=1pt] (4){5}
        (17.856, 88.361) node[draw, fill=gray!20, minimum size=1pt] (5){6}
        (96.902, 50.93) node[draw, fill=gray!20, minimum size=1pt] (6){7}
        (55.599, 13.246) node[draw, fill=gray!20, minimum size=1pt] (7){8};
      \begin{scope}[->]
        \draw[color=gray, line width=1pt] (0) to (1);
        \draw[color=black, line width=3pt] (0) to (2);
        \draw[color=gray, line width=1pt] (0) to (3);
        \draw[color=gray, line width=1pt] (0) to (4);
        \draw[color=gray, line width=1pt] (0) to (5);
        \draw[color=gray, line width=1pt] (0) to (6);
        \draw[color=gray, line width=1pt] (0) to (7);
        \draw[color=gray, line width=1pt] (1) to (0);
        \draw[color=gray, line width=1pt] (1) to (2);
        \draw[color=gray, line width=1pt] (1) to (3);
        \draw[color=black, line width=3pt] (1) to (4);
        \draw[color=gray, line width=1pt] (1) to (5);
        \draw[color=gray, line width=1pt] (1) to (6);
        \draw[color=gray, line width=1pt] (1) to (7);
        \draw[color=gray, line width=1pt] (2) to (0);
        \draw[color=gray, line width=1pt] (2) to (1);
        \draw[color=black, line width=3pt] (2) to (3);
        \draw[color=gray, line width=1pt] (2) to (4);
        \draw[color=gray, line width=1pt] (2) to (5);
        \draw[color=gray, line width=1pt] (2) to (6);
        \draw[color=gray, line width=1pt] (2) to (7);
        \draw[color=gray, line width=1pt] (3) to (0);
        \draw[color=gray, line width=1pt] (3) to (1);
        \draw[color=gray, line width=1pt] (3) to (2);
        \draw[color=gray, line width=1pt] (3) to (4);
        \draw[color=gray, line width=1pt] (3) to (5);
        \draw[color=gray, line width=1pt] (3) to (6);
        \draw[color=black, line width=3pt] (3) to (7);
        \draw[color=gray, line width=1pt] (4) to (0);
        \draw[color=gray, line width=1pt] (4) to (1);
        \draw[color=gray, line width=1pt] (4) to (2);
        \draw[color=gray, line width=1pt] (4) to (3);
        \draw[color=black, line width=3pt] (4) to (5);
        \draw[color=gray, line width=1pt] (4) to (6);
        \draw[color=gray, line width=1pt] (4) to (7);
        \draw[color=gray, line width=1pt] (5) to (0);
        \draw[color=gray, line width=1pt] (5) to (1);
        \draw[color=gray, line width=1pt] (5) to (2);
        \draw[color=gray, line width=1pt] (5) to (3);
        \draw[color=gray, line width=1pt] (5) to (4);
        \draw[color=black, line width=3pt] (5) to (6);
        \draw[color=gray, line width=1pt] (5) to (7);
        \draw[color=black, line width=3pt] (6) to (0);
        \draw[color=gray, line width=1pt] (6) to (1);
        \draw[color=gray, line width=1pt] (6) to (2);
        \draw[color=gray, line width=1pt] (6) to (3);
        \draw[color=gray, line width=1pt] (6) to (4);
        \draw[color=gray, line width=1pt] (6) to (5);
        \draw[color=gray, line width=1pt] (6) to (7);
        \draw[color=gray, line width=1pt] (7) to (0);
        \draw[color=black, line width=3pt] (7) to (1);
        \draw[color=gray, line width=1pt] (7) to (2);
        \draw[color=gray, line width=1pt] (7) to (3);
        \draw[color=gray, line width=1pt] (7) to (4);
        \draw[color=gray, line width=1pt] (7) to (5);
        \draw[color=gray, line width=1pt] (7) to (6);
      \end{scope}
    \end{tikzpicture}}}~\fbox{\adjustbox{width=0.33\textwidth}{{\centering $\ell$-\texttt{trim}}
  \begin{tikzpicture}[scale=0.1]
      \draw
        (99.964, 29.443) node[draw, fill=gray!20, minimum size=1pt] (0){1}
        (8.374, 34.1) node[draw, fill=gray!20, minimum size=1pt] (1){2}
        (89.791, 18.66) node[draw, fill=gray!20, minimum size=1pt] (2){3}
        (60.064, 34.381) node[draw, fill=gray!20, minimum size=1pt] (3){4}
        (22.225, 50.676) node[draw, fill=gray!20, minimum size=1pt] (4){5}
        (17.856, 88.361) node[draw, fill=gray!20, minimum size=1pt] (5){6}
        (96.902, 50.93) node[draw, fill=gray!20, minimum size=1pt] (6){7}
        (55.599, 13.246) node[draw, fill=gray!20, minimum size=1pt] (7){8};
      \begin{scope}[->]
        \draw[color=gray, line width=1pt] (0) to (1);
        \draw[color=gray, line width=1pt] (0) to (2);
        \draw[color=gray, line width=1pt] (0) to (3);
        \draw[color=gray, line width=1pt] (0) to (4);
        \draw[color=gray, line width=1pt] (0) to (5);
        \draw[color=black, line width=3pt] (0) to (6);
        \draw[color=gray, line width=1pt] (0) to (7);
        \draw[color=gray, line width=1pt] (1) to (0);
        \draw[color=gray, line width=1pt] (1) to (2);
        \draw[color=gray, line width=1pt] (1) to (3);
        \draw[color=gray, line width=1pt] (1) to (4);
        \draw[color=gray, line width=1pt] (1) to (5);
        \draw[color=gray, line width=1pt] (1) to (6);
        \draw[color=black, line width=3pt] (1) to (7);
        \draw[color=black, line width=3pt] (2) to (0);
        \draw[color=gray, line width=1pt] (2) to (1);
        \draw[color=gray, line width=1pt] (2) to (3);
        \draw[color=gray, line width=1pt] (2) to (4);
        \draw[color=gray, line width=1pt] (2) to (5);
        \draw[color=gray, line width=1pt] (2) to (6);
        \draw[color=gray, line width=1pt] (2) to (7);
        \draw[color=gray, line width=1pt] (3) to (0);
        \draw[color=gray, line width=1pt] (3) to (1);
        \draw[color=black, line width=3pt] (3) to (2);
        \draw[color=gray, line width=1pt] (3) to (4);
        \draw[color=gray, line width=1pt] (3) to (5);
        \draw[color=gray, line width=1pt] (3) to (6);
        \draw[color=gray, line width=1pt] (3) to (7);
        \draw[color=gray, line width=1pt] (4) to (0);
        \draw[color=black, line width=3pt] (4) to (1);
        \draw[color=gray, line width=1pt] (4) to (2);
        \draw[color=gray, line width=1pt] (4) to (3);
        \draw[color=gray, line width=1pt] (4) to (5);
        \draw[color=gray, line width=1pt] (4) to (6);
        \draw[color=gray, line width=1pt] (4) to (7);
        \draw[color=gray, line width=1pt] (5) to (0);
        \draw[color=gray, line width=1pt] (5) to (1);
        \draw[color=gray, line width=1pt] (5) to (2);
        \draw[color=gray, line width=1pt] (5) to (3);
        \draw[color=black, line width=3pt] (5) to (4);
        \draw[color=gray, line width=1pt] (5) to (6);
        \draw[color=gray, line width=1pt] (5) to (7);
        \draw[color=gray, line width=1pt] (6) to (0);
        \draw[color=gray, line width=1pt] (6) to (1);
        \draw[color=gray, line width=1pt] (6) to (2);
        \draw[color=gray, line width=1pt] (6) to (3);
        \draw[color=gray, line width=1pt] (6) to (4);
        \draw[color=black, line width=3pt] (6) to (5);
        \draw[color=gray, line width=1pt] (6) to (7);
        \draw[color=gray, line width=1pt] (7) to (0);
        \draw[color=gray, line width=1pt] (7) to (1);
        \draw[color=gray, line width=1pt] (7) to (2);
        \draw[color=black, line width=3pt] (7) to (3);
        \draw[color=gray, line width=1pt] (7) to (4);
        \draw[color=gray, line width=1pt] (7) to (5);
        \draw[color=gray, line width=1pt] (7) to (6);
      \end{scope}
    \end{tikzpicture}}}~\fbox{\adjustbox{width=0.33\textwidth}{{\centering $\ell$-\texttt{max}}
  \begin{tikzpicture}[scale=0.1]
      \draw
        (99.964, 29.443) node[draw, fill=gray!20, minimum size=1pt] (0){1}
        (8.374, 34.1) node[draw, fill=gray!20, minimum size=1pt] (1){2}
        (89.791, 18.66) node[draw, fill=gray!20, minimum size=1pt] (2){3}
        (60.064, 34.381) node[draw, fill=gray!20, minimum size=1pt] (3){4}
        (22.225, 50.676) node[draw, fill=gray!20, minimum size=1pt] (4){5}
        (17.856, 88.361) node[draw, fill=gray!20, minimum size=1pt] (5){6}
        (96.902, 50.93) node[draw, fill=gray!20, minimum size=1pt] (6){7}
        (55.599, 13.246) node[draw, fill=gray!20, minimum size=1pt] (7){8};
      \begin{scope}[->]
        \draw[color=gray, line width=1pt] (0) to (1);
        \draw[color=gray, line width=1pt] (0) to (2);
        \draw[color=gray, line width=1pt] (0) to (3);
        \draw[color=gray, line width=1pt] (0) to (4);
        \draw[color=gray, line width=1pt] (0) to (5);
        \draw[color=black, line width=3pt] (0) to (6);
        \draw[color=gray, line width=1pt] (0) to (7);
        \draw[color=gray, line width=1pt] (1) to (0);
        \draw[color=gray, line width=1pt] (1) to (2);
        \draw[color=gray, line width=1pt] (1) to (3);
        \draw[color=gray, line width=1pt] (1) to (4);
        \draw[color=black, line width=3pt] (1) to (5);
        \draw[color=gray, line width=1pt] (1) to (6);
        \draw[color=gray, line width=1pt] (1) to (7);
        \draw[color=gray, line width=1pt] (2) to (0);
        \draw[color=gray, line width=1pt] (2) to (1);
        \draw[color=gray, line width=1pt] (2) to (3);
        \draw[color=gray, line width=1pt] (2) to (4);
        \draw[color=gray, line width=1pt] (2) to (5);
        \draw[color=gray, line width=1pt] (2) to (6);
        \draw[color=black, line width=3pt] (2) to (7);
        \draw[color=black, line width=3pt] (3) to (0);
        \draw[color=gray, line width=1pt] (3) to (1);
        \draw[color=gray, line width=1pt] (3) to (2);
        \draw[color=gray, line width=1pt] (3) to (4);
        \draw[color=gray, line width=1pt] (3) to (5);
        \draw[color=gray, line width=1pt] (3) to (6);
        \draw[color=gray, line width=1pt] (3) to (7);
        \draw[color=gray, line width=1pt] (4) to (0);
        \draw[color=gray, line width=1pt] (4) to (1);
        \draw[color=gray, line width=1pt] (4) to (2);
        \draw[color=black, line width=3pt] (4) to (3);
        \draw[color=gray, line width=1pt] (4) to (5);
        \draw[color=gray, line width=1pt] (4) to (6);
        \draw[color=gray, line width=1pt] (4) to (7);
        \draw[color=gray, line width=1pt] (5) to (0);
        \draw[color=gray, line width=1pt] (5) to (1);
        \draw[color=gray, line width=1pt] (5) to (2);
        \draw[color=gray, line width=1pt] (5) to (3);
        \draw[color=black, line width=3pt] (5) to (4);
        \draw[color=gray, line width=1pt] (5) to (6);
        \draw[color=gray, line width=1pt] (5) to (7);
        \draw[color=gray, line width=1pt] (6) to (0);
        \draw[color=gray, line width=1pt] (6) to (1);
        \draw[color=black, line width=3pt] (6) to (2);
        \draw[color=gray, line width=1pt] (6) to (3);
        \draw[color=gray, line width=1pt] (6) to (4);
        \draw[color=gray, line width=1pt] (6) to (5);
        \draw[color=gray, line width=1pt] (6) to (7);
        \draw[color=gray, line width=1pt] (7) to (0);
        \draw[color=black, line width=3pt] (7) to (1);
        \draw[color=gray, line width=1pt] (7) to (2);
        \draw[color=gray, line width=1pt] (7) to (3);
        \draw[color=gray, line width=1pt] (7) to (4);
        \draw[color=gray, line width=1pt] (7) to (5);
        \draw[color=gray, line width=1pt] (7) to (6);
      \end{scope}
    \end{tikzpicture}}}
    \caption{Solutions for the TSP example that minimize  different $\ell$ ordered measures.\label{fig:tsp_solsL}}
\end{center}
\end{figure}
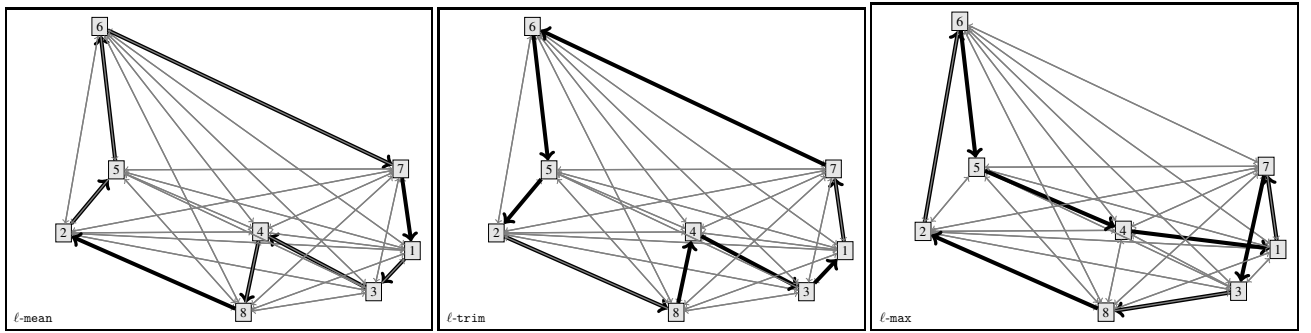

\begin{figure}[h]
\begin{center}
\fbox{\adjustbox{width=0.33\textwidth}{{\centering $\mathfrak{b}$-\texttt{mean}}
  \begin{tikzpicture}[scale=0.1]
      \draw
        (99.964, 29.443) node[draw, fill=gray!20, minimum size=1pt] (0){1}
        (8.374, 34.1) node[draw, fill=gray!20, minimum size=1pt] (1){2}
        (89.791, 18.66) node[draw, fill=gray!20, minimum size=1pt] (2){3}
        (60.064, 34.381) node[draw, fill=gray!20, minimum size=1pt] (3){4}
        (22.225, 50.676) node[draw, fill=gray!20, minimum size=1pt] (4){5}
        (17.856, 88.361) node[draw, fill=gray!20, minimum size=1pt] (5){6}
        (96.902, 50.93) node[draw, fill=gray!20, minimum size=1pt] (6){7}
        (55.599, 13.246) node[draw, fill=gray!20, minimum size=1pt] (7){8};
      \begin{scope}[->]
        \draw[color=gray, line width=1pt] (0) to (1);
        \draw[color=gray, line width=1pt] (0) to (2);
        \draw[color=black, line width=3pt] (0) to (3);
        \draw[color=gray, line width=1pt] (0) to (4);
        \draw[color=gray, line width=1pt] (0) to (5);
        \draw[color=gray, line width=1pt] (0) to (6);
        \draw[color=gray, line width=1pt] (0) to (7);
        \draw[color=gray, line width=1pt] (1) to (0);
        \draw[color=gray, line width=1pt] (1) to (2);
        \draw[color=gray, line width=1pt] (1) to (3);
        \draw[color=gray, line width=1pt] (1) to (4);
        \draw[color=gray, line width=1pt] (1) to (5);
        \draw[color=gray, line width=1pt] (1) to (6);
        \draw[color=black, line width=3pt] (1) to (7);
        \draw[color=gray, line width=1pt] (2) to (0);
        \draw[color=gray, line width=1pt] (2) to (1);
        \draw[color=gray, line width=1pt] (2) to (3);
        \draw[color=gray, line width=1pt] (2) to (4);
        \draw[color=gray, line width=1pt] (2) to (5);
        \draw[color=black, line width=3pt] (2) to (6);
        \draw[color=gray, line width=1pt] (2) to (7);
        \draw[color=gray, line width=1pt] (3) to (0);
        \draw[color=gray, line width=1pt] (3) to (1);
        \draw[color=gray, line width=1pt] (3) to (2);
        \draw[color=black, line width=3pt] (3) to (4);
        \draw[color=gray, line width=1pt] (3) to (5);
        \draw[color=gray, line width=1pt] (3) to (6);
        \draw[color=gray, line width=1pt] (3) to (7);
        \draw[color=gray, line width=1pt] (4) to (0);
        \draw[color=gray, line width=1pt] (4) to (1);
        \draw[color=gray, line width=1pt] (4) to (2);
        \draw[color=gray, line width=1pt] (4) to (3);
        \draw[color=black, line width=3pt] (4) to (5);
        \draw[color=gray, line width=1pt] (4) to (6);
        \draw[color=gray, line width=1pt] (4) to (7);
        \draw[color=gray, line width=1pt] (5) to (0);
        \draw[color=black, line width=3pt] (5) to (1);
        \draw[color=gray, line width=1pt] (5) to (2);
        \draw[color=gray, line width=1pt] (5) to (3);
        \draw[color=gray, line width=1pt] (5) to (4);
        \draw[color=gray, line width=1pt] (5) to (6);
        \draw[color=gray, line width=1pt] (5) to (7);
        \draw[color=black, line width=3pt] (6) to (0);
        \draw[color=gray, line width=1pt] (6) to (1);
        \draw[color=gray, line width=1pt] (6) to (2);
        \draw[color=gray, line width=1pt] (6) to (3);
        \draw[color=gray, line width=1pt] (6) to (4);
        \draw[color=gray, line width=1pt] (6) to (5);
        \draw[color=gray, line width=1pt] (6) to (7);
        \draw[color=gray, line width=1pt] (7) to (0);
        \draw[color=gray, line width=1pt] (7) to (1);
        \draw[color=black, line width=3pt] (7) to (2);
        \draw[color=gray, line width=1pt] (7) to (3);
        \draw[color=gray, line width=1pt] (7) to (4);
        \draw[color=gray, line width=1pt] (7) to (5);
        \draw[color=gray, line width=1pt] (7) to (6);
      \end{scope}
    \end{tikzpicture}}}~\fbox{\adjustbox{width=0.33\textwidth}{{\centering $\mathfrak{b}$-\texttt{trim}}
  \begin{tikzpicture}[scale=0.1]
      \draw
        (99.964, 29.443) node[draw, fill=gray!20, minimum size=1pt] (0){1}
        (8.374, 34.1) node[draw, fill=gray!20, minimum size=1pt] (1){2}
        (89.791, 18.66) node[draw, fill=gray!20, minimum size=1pt] (2){3}
        (60.064, 34.381) node[draw, fill=gray!20, minimum size=1pt] (3){4}
        (22.225, 50.676) node[draw, fill=gray!20, minimum size=1pt] (4){5}
        (17.856, 88.361) node[draw, fill=gray!20, minimum size=1pt] (5){6}
        (96.902, 50.93) node[draw, fill=gray!20, minimum size=1pt] (6){7}
        (55.599, 13.246) node[draw, fill=gray!20, minimum size=1pt] (7){8};
      \begin{scope}[->]
        \draw[color=gray, line width=1pt] (0) to (1);
        \draw[color=black, line width=3pt] (0) to (2);
        \draw[color=gray, line width=1pt] (0) to (3);
        \draw[color=gray, line width=1pt] (0) to (4);
        \draw[color=gray, line width=1pt] (0) to (5);
        \draw[color=gray, line width=1pt] (0) to (6);
        \draw[color=gray, line width=1pt] (0) to (7);
        \draw[color=gray, line width=1pt] (1) to (0);
        \draw[color=gray, line width=1pt] (1) to (2);
        \draw[color=gray, line width=1pt] (1) to (3);
        \draw[color=black, line width=3pt] (1) to (4);
        \draw[color=gray, line width=1pt] (1) to (5);
        \draw[color=gray, line width=1pt] (1) to (6);
        \draw[color=gray, line width=1pt] (1) to (7);
        \draw[color=gray, line width=1pt] (2) to (0);
        \draw[color=gray, line width=1pt] (2) to (1);
        \draw[color=gray, line width=1pt] (2) to (3);
        \draw[color=gray, line width=1pt] (2) to (4);
        \draw[color=gray, line width=1pt] (2) to (5);
        \draw[color=gray, line width=1pt] (2) to (6);
        \draw[color=black, line width=3pt] (2) to (7);
        \draw[color=gray, line width=1pt] (3) to (0);
        \draw[color=gray, line width=1pt] (3) to (1);
        \draw[color=gray, line width=1pt] (3) to (2);
        \draw[color=gray, line width=1pt] (3) to (4);
        \draw[color=gray, line width=1pt] (3) to (5);
        \draw[color=black, line width=3pt] (3) to (6);
        \draw[color=gray, line width=1pt] (3) to (7);
        \draw[color=gray, line width=1pt] (4) to (0);
        \draw[color=gray, line width=1pt] (4) to (1);
        \draw[color=gray, line width=1pt] (4) to (2);
        \draw[color=gray, line width=1pt] (4) to (3);
        \draw[color=black, line width=3pt] (4) to (5);
        \draw[color=gray, line width=1pt] (4) to (6);
        \draw[color=gray, line width=1pt] (4) to (7);
        \draw[color=gray, line width=1pt] (5) to (0);
        \draw[color=gray, line width=1pt] (5) to (1);
        \draw[color=gray, line width=1pt] (5) to (2);
        \draw[color=black, line width=3pt] (5) to (3);
        \draw[color=gray, line width=1pt] (5) to (4);
        \draw[color=gray, line width=1pt] (5) to (6);
        \draw[color=gray, line width=1pt] (5) to (7);
        \draw[color=black, line width=3pt] (6) to (0);
        \draw[color=gray, line width=1pt] (6) to (1);
        \draw[color=gray, line width=1pt] (6) to (2);
        \draw[color=gray, line width=1pt] (6) to (3);
        \draw[color=gray, line width=1pt] (6) to (4);
        \draw[color=gray, line width=1pt] (6) to (5);
        \draw[color=gray, line width=1pt] (6) to (7);
        \draw[color=gray, line width=1pt] (7) to (0);
        \draw[color=black, line width=3pt] (7) to (1);
        \draw[color=gray, line width=1pt] (7) to (2);
        \draw[color=gray, line width=1pt] (7) to (3);
        \draw[color=gray, line width=1pt] (7) to (4);
        \draw[color=gray, line width=1pt] (7) to (5);
        \draw[color=gray, line width=1pt] (7) to (6);
      \end{scope}
    \end{tikzpicture}}}~\fbox{\adjustbox{width=0.33\textwidth}{{\centering $\mathfrak{b}$-\texttt{max}}
  \begin{tikzpicture}[scale=0.1]
      \draw
        (99.964, 29.443) node[draw, fill=gray!20, minimum size=1pt] (0){1}
        (8.374, 34.1) node[draw, fill=gray!20, minimum size=1pt] (1){2}
        (89.791, 18.66) node[draw, fill=gray!20, minimum size=1pt] (2){3}
        (60.064, 34.381) node[draw, fill=gray!20, minimum size=1pt] (3){4}
        (22.225, 50.676) node[draw, fill=gray!20, minimum size=1pt] (4){5}
        (17.856, 88.361) node[draw, fill=gray!20, minimum size=1pt] (5){6}
        (96.902, 50.93) node[draw, fill=gray!20, minimum size=1pt] (6){7}
        (55.599, 13.246) node[draw, fill=gray!20, minimum size=1pt] (7){8};
      \begin{scope}[->]
        \draw[color=gray, line width=1pt] (0) to (1);
        \draw[color=gray, line width=1pt] (0) to (2);
        \draw[color=black, line width=3pt] (0) to (3);
        \draw[color=gray, line width=1pt] (0) to (4);
        \draw[color=gray, line width=1pt] (0) to (5);
        \draw[color=gray, line width=1pt] (0) to (6);
        \draw[color=gray, line width=1pt] (0) to (7);
        \draw[color=gray, line width=1pt] (1) to (0);
        \draw[color=gray, line width=1pt] (1) to (2);
        \draw[color=gray, line width=1pt] (1) to (3);
        \draw[color=gray, line width=1pt] (1) to (4);
        \draw[color=gray, line width=1pt] (1) to (5);
        \draw[color=gray, line width=1pt] (1) to (6);
        \draw[color=black, line width=3pt] (1) to (7);
        \draw[color=gray, line width=1pt] (2) to (0);
        \draw[color=gray, line width=1pt] (2) to (1);
        \draw[color=gray, line width=1pt] (2) to (3);
        \draw[color=gray, line width=1pt] (2) to (4);
        \draw[color=gray, line width=1pt] (2) to (5);
        \draw[color=black, line width=3pt] (2) to (6);
        \draw[color=gray, line width=1pt] (2) to (7);
        \draw[color=gray, line width=1pt] (3) to (0);
        \draw[color=gray, line width=1pt] (3) to (1);
        \draw[color=gray, line width=1pt] (3) to (2);
        \draw[color=black, line width=3pt] (3) to (4);
        \draw[color=gray, line width=1pt] (3) to (5);
        \draw[color=gray, line width=1pt] (3) to (6);
        \draw[color=gray, line width=1pt] (3) to (7);
        \draw[color=gray, line width=1pt] (4) to (0);
        \draw[color=gray, line width=1pt] (4) to (1);
        \draw[color=gray, line width=1pt] (4) to (2);
        \draw[color=gray, line width=1pt] (4) to (3);
        \draw[color=black, line width=3pt] (4) to (5);
        \draw[color=gray, line width=1pt] (4) to (6);
        \draw[color=gray, line width=1pt] (4) to (7);
        \draw[color=gray, line width=1pt] (5) to (0);
        \draw[color=black, line width=3pt] (5) to (1);
        \draw[color=gray, line width=1pt] (5) to (2);
        \draw[color=gray, line width=1pt] (5) to (3);
        \draw[color=gray, line width=1pt] (5) to (4);
        \draw[color=gray, line width=1pt] (5) to (6);
        \draw[color=gray, line width=1pt] (5) to (7);
        \draw[color=black, line width=3pt] (6) to (0);
        \draw[color=gray, line width=1pt] (6) to (1);
        \draw[color=gray, line width=1pt] (6) to (2);
        \draw[color=gray, line width=1pt] (6) to (3);
        \draw[color=gray, line width=1pt] (6) to (4);
        \draw[color=gray, line width=1pt] (6) to (5);
        \draw[color=gray, line width=1pt] (6) to (7);
        \draw[color=gray, line width=1pt] (7) to (0);
        \draw[color=gray, line width=1pt] (7) to (1);
        \draw[color=black, line width=3pt] (7) to (2);
        \draw[color=gray, line width=1pt] (7) to (3);
        \draw[color=gray, line width=1pt] (7) to (4);
        \draw[color=gray, line width=1pt] (7) to (5);
        \draw[color=gray, line width=1pt] (7) to (6);
      \end{scope}
    \end{tikzpicture}}}
    \caption{Solutions for the TSP example that minimize  different $\mathfrak{b}$ ordered measures.\label{fig:tsp_solsN}}
\end{center}
\end{figure}

\begin{figure}[h]
\begin{center}
\fbox{\adjustbox{width=0.33\textwidth}{{\centering $\QO$-\texttt{mean}}
  \begin{tikzpicture}[scale=0.1]
      \draw
        (99.964, 29.443) node[draw, fill=gray!20, minimum size=1pt] (0){1}
        (8.374, 34.1) node[draw, fill=gray!20, minimum size=1pt] (1){2}
        (89.791, 18.66) node[draw, fill=gray!20, minimum size=1pt] (2){3}
        (60.064, 34.381) node[draw, fill=gray!20, minimum size=1pt] (3){4}
        (22.225, 50.676) node[draw, fill=gray!20, minimum size=1pt] (4){5}
        (17.856, 88.361) node[draw, fill=gray!20, minimum size=1pt] (5){6}
        (96.902, 50.93) node[draw, fill=gray!20, minimum size=1pt] (6){7}
        (55.599, 13.246) node[draw, fill=gray!20, minimum size=1pt] (7){8};
      \begin{scope}[->]
        \draw[color=gray, line width=1pt] (0) to (1);
        \draw[color=gray, line width=1pt] (0) to (2);
        \draw[color=black, line width=3pt] (0) to (3);
        \draw[color=gray, line width=1pt] (0) to (4);
        \draw[color=gray, line width=1pt] (0) to (5);
        \draw[color=gray, line width=1pt] (0) to (6);
        \draw[color=gray, line width=1pt] (0) to (7);
        \draw[color=gray, line width=1pt] (1) to (0);
        \draw[color=gray, line width=1pt] (1) to (2);
        \draw[color=gray, line width=1pt] (1) to (3);
        \draw[color=gray, line width=1pt] (1) to (4);
        \draw[color=gray, line width=1pt] (1) to (5);
        \draw[color=gray, line width=1pt] (1) to (6);
        \draw[color=black, line width=3pt] (1) to (7);
        \draw[color=gray, line width=1pt] (2) to (0);
        \draw[color=gray, line width=1pt] (2) to (1);
        \draw[color=gray, line width=1pt] (2) to (3);
        \draw[color=gray, line width=1pt] (2) to (4);
        \draw[color=gray, line width=1pt] (2) to (5);
        \draw[color=black, line width=3pt] (2) to (6);
        \draw[color=gray, line width=1pt] (2) to (7);
        \draw[color=gray, line width=1pt] (3) to (0);
        \draw[color=gray, line width=1pt] (3) to (1);
        \draw[color=gray, line width=1pt] (3) to (2);
        \draw[color=black, line width=3pt] (3) to (4);
        \draw[color=gray, line width=1pt] (3) to (5);
        \draw[color=gray, line width=1pt] (3) to (6);
        \draw[color=gray, line width=1pt] (3) to (7);
        \draw[color=gray, line width=1pt] (4) to (0);
        \draw[color=gray, line width=1pt] (4) to (1);
        \draw[color=gray, line width=1pt] (4) to (2);
        \draw[color=gray, line width=1pt] (4) to (3);
        \draw[color=black, line width=3pt] (4) to (5);
        \draw[color=gray, line width=1pt] (4) to (6);
        \draw[color=gray, line width=1pt] (4) to (7);
        \draw[color=gray, line width=1pt] (5) to (0);
        \draw[color=black, line width=3pt] (5) to (1);
        \draw[color=gray, line width=1pt] (5) to (2);
        \draw[color=gray, line width=1pt] (5) to (3);
        \draw[color=gray, line width=1pt] (5) to (4);
        \draw[color=gray, line width=1pt] (5) to (6);
        \draw[color=gray, line width=1pt] (5) to (7);
        \draw[color=black, line width=3pt] (6) to (0);
        \draw[color=gray, line width=1pt] (6) to (1);
        \draw[color=gray, line width=1pt] (6) to (2);
        \draw[color=gray, line width=1pt] (6) to (3);
        \draw[color=gray, line width=1pt] (6) to (4);
        \draw[color=gray, line width=1pt] (6) to (5);
        \draw[color=gray, line width=1pt] (6) to (7);
        \draw[color=gray, line width=1pt] (7) to (0);
        \draw[color=gray, line width=1pt] (7) to (1);
        \draw[color=black, line width=3pt] (7) to (2);
        \draw[color=gray, line width=1pt] (7) to (3);
        \draw[color=gray, line width=1pt] (7) to (4);
        \draw[color=gray, line width=1pt] (7) to (5);
        \draw[color=gray, line width=1pt] (7) to (6);
      \end{scope}
    \end{tikzpicture}}}~\fbox{\adjustbox{width=0.33\textwidth}{{\centering $\QO$-\texttt{trim}}
  \begin{tikzpicture}[scale=0.1]
      \draw
        (99.964, 29.443) node[draw, fill=gray!20, minimum size=1pt] (0){1}
        (8.374, 34.1) node[draw, fill=gray!20, minimum size=1pt] (1){2}
        (89.791, 18.66) node[draw, fill=gray!20, minimum size=1pt] (2){3}
        (60.064, 34.381) node[draw, fill=gray!20, minimum size=1pt] (3){4}
        (22.225, 50.676) node[draw, fill=gray!20, minimum size=1pt] (4){5}
        (17.856, 88.361) node[draw, fill=gray!20, minimum size=1pt] (5){6}
        (96.902, 50.93) node[draw, fill=gray!20, minimum size=1pt] (6){7}
        (55.599, 13.246) node[draw, fill=gray!20, minimum size=1pt] (7){8};
      \begin{scope}[->]
        \draw[color=gray, line width=1pt] (0) to (1);
        \draw[color=gray, line width=1pt] (0) to (2);
        \draw[color=gray, line width=1pt] (0) to (3);
        \draw[color=gray, line width=1pt] (0) to (4);
        \draw[color=gray, line width=1pt] (0) to (5);
        \draw[color=black, line width=3pt] (0) to (6);
        \draw[color=gray, line width=1pt] (0) to (7);
        \draw[color=gray, line width=1pt] (1) to (0);
        \draw[color=gray, line width=1pt] (1) to (2);
        \draw[color=gray, line width=1pt] (1) to (3);
        \draw[color=gray, line width=1pt] (1) to (4);
        \draw[color=gray, line width=1pt] (1) to (5);
        \draw[color=gray, line width=1pt] (1) to (6);
        \draw[color=black, line width=3pt] (1) to (7);
        \draw[color=black, line width=3pt] (2) to (0);
        \draw[color=gray, line width=1pt] (2) to (1);
        \draw[color=gray, line width=1pt] (2) to (3);
        \draw[color=gray, line width=1pt] (2) to (4);
        \draw[color=gray, line width=1pt] (2) to (5);
        \draw[color=gray, line width=1pt] (2) to (6);
        \draw[color=gray, line width=1pt] (2) to (7);
        \draw[color=gray, line width=1pt] (3) to (0);
        \draw[color=gray, line width=1pt] (3) to (1);
        \draw[color=black, line width=3pt] (3) to (2);
        \draw[color=gray, line width=1pt] (3) to (4);
        \draw[color=gray, line width=1pt] (3) to (5);
        \draw[color=gray, line width=1pt] (3) to (6);
        \draw[color=gray, line width=1pt] (3) to (7);
        \draw[color=gray, line width=1pt] (4) to (0);
        \draw[color=black, line width=3pt] (4) to (1);
        \draw[color=gray, line width=1pt] (4) to (2);
        \draw[color=gray, line width=1pt] (4) to (3);
        \draw[color=gray, line width=1pt] (4) to (5);
        \draw[color=gray, line width=1pt] (4) to (6);
        \draw[color=gray, line width=1pt] (4) to (7);
        \draw[color=gray, line width=1pt] (5) to (0);
        \draw[color=gray, line width=1pt] (5) to (1);
        \draw[color=gray, line width=1pt] (5) to (2);
        \draw[color=gray, line width=1pt] (5) to (3);
        \draw[color=black, line width=3pt] (5) to (4);
        \draw[color=gray, line width=1pt] (5) to (6);
        \draw[color=gray, line width=1pt] (5) to (7);
        \draw[color=gray, line width=1pt] (6) to (0);
        \draw[color=gray, line width=1pt] (6) to (1);
        \draw[color=gray, line width=1pt] (6) to (2);
        \draw[color=gray, line width=1pt] (6) to (3);
        \draw[color=gray, line width=1pt] (6) to (4);
        \draw[color=black, line width=3pt] (6) to (5);
        \draw[color=gray, line width=1pt] (6) to (7);
        \draw[color=gray, line width=1pt] (7) to (0);
        \draw[color=gray, line width=1pt] (7) to (1);
        \draw[color=gray, line width=1pt] (7) to (2);
        \draw[color=black, line width=3pt] (7) to (3);
        \draw[color=gray, line width=1pt] (7) to (4);
        \draw[color=gray, line width=1pt] (7) to (5);
        \draw[color=gray, line width=1pt] (7) to (6);
      \end{scope}
    \end{tikzpicture}}}~\fbox{\adjustbox{width=0.33\textwidth}{{\centering $\QO$-\texttt{max}}
  \begin{tikzpicture}[scale=0.1]
      \draw
        (99.964, 29.443) node[draw, fill=gray!20, minimum size=1pt] (0){1}
        (8.374, 34.1) node[draw, fill=gray!20, minimum size=1pt] (1){2}
        (89.791, 18.66) node[draw, fill=gray!20, minimum size=1pt] (2){3}
        (60.064, 34.381) node[draw, fill=gray!20, minimum size=1pt] (3){4}
        (22.225, 50.676) node[draw, fill=gray!20, minimum size=1pt] (4){5}
        (17.856, 88.361) node[draw, fill=gray!20, minimum size=1pt] (5){6}
        (96.902, 50.93) node[draw, fill=gray!20, minimum size=1pt] (6){7}
        (55.599, 13.246) node[draw, fill=gray!20, minimum size=1pt] (7){8};
      \begin{scope}[->]
        \draw[color=gray, line width=1pt] (0) to (1);
        \draw[color=gray, line width=1pt] (0) to (2);
        \draw[color=gray, line width=1pt] (0) to (3);
        \draw[color=gray, line width=1pt] (0) to (4);
        \draw[color=gray, line width=1pt] (0) to (5);
        \draw[color=gray, line width=1pt] (0) to (6);
        \draw[color=black, line width=3pt] (0) to (7);
        \draw[color=gray, line width=1pt] (1) to (0);
        \draw[color=gray, line width=1pt] (1) to (2);
        \draw[color=black, line width=3pt] (1) to (3);
        \draw[color=gray, line width=1pt] (1) to (4);
        \draw[color=gray, line width=1pt] (1) to (5);
        \draw[color=gray, line width=1pt] (1) to (6);
        \draw[color=gray, line width=1pt] (1) to (7);
        \draw[color=gray, line width=1pt] (2) to (0);
        \draw[color=gray, line width=1pt] (2) to (1);
        \draw[color=gray, line width=1pt] (2) to (3);
        \draw[color=gray, line width=1pt] (2) to (4);
        \draw[color=gray, line width=1pt] (2) to (5);
        \draw[color=black, line width=3pt] (2) to (6);
        \draw[color=gray, line width=1pt] (2) to (7);
        \draw[color=gray, line width=1pt] (3) to (0);
        \draw[color=gray, line width=1pt] (3) to (1);
        \draw[color=black, line width=3pt] (3) to (2);
        \draw[color=gray, line width=1pt] (3) to (4);
        \draw[color=gray, line width=1pt] (3) to (5);
        \draw[color=gray, line width=1pt] (3) to (6);
        \draw[color=gray, line width=1pt] (3) to (7);
        \draw[color=gray, line width=1pt] (4) to (0);
        \draw[color=gray, line width=1pt] (4) to (1);
        \draw[color=gray, line width=1pt] (4) to (2);
        \draw[color=gray, line width=1pt] (4) to (3);
        \draw[color=black, line width=3pt] (4) to (5);
        \draw[color=gray, line width=1pt] (4) to (6);
        \draw[color=gray, line width=1pt] (4) to (7);
        \draw[color=gray, line width=1pt] (5) to (0);
        \draw[color=black, line width=3pt] (5) to (1);
        \draw[color=gray, line width=1pt] (5) to (2);
        \draw[color=gray, line width=1pt] (5) to (3);
        \draw[color=gray, line width=1pt] (5) to (4);
        \draw[color=gray, line width=1pt] (5) to (6);
        \draw[color=gray, line width=1pt] (5) to (7);
        \draw[color=black, line width=3pt] (6) to (0);
        \draw[color=gray, line width=1pt] (6) to (1);
        \draw[color=gray, line width=1pt] (6) to (2);
        \draw[color=gray, line width=1pt] (6) to (3);
        \draw[color=gray, line width=1pt] (6) to (4);
        \draw[color=gray, line width=1pt] (6) to (5);
        \draw[color=gray, line width=1pt] (6) to (7);
        \draw[color=gray, line width=1pt] (7) to (0);
        \draw[color=gray, line width=1pt] (7) to (1);
        \draw[color=gray, line width=1pt] (7) to (2);
        \draw[color=gray, line width=1pt] (7) to (3);
        \draw[color=black, line width=3pt] (7) to (4);
        \draw[color=gray, line width=1pt] (7) to (5);
        \draw[color=gray, line width=1pt] (7) to (6);
      \end{scope}
    \end{tikzpicture}}}
    \caption{Solutions for the TSP example that minimize  different $\QO$ ordered measures..\label{fig:tsp_solsQ}}
\end{center}
\end{figure}

Note that even though the combinatorial complexity of this toy example is limited, in almost all cases, the solutions obtained with the different methodologies are different. For instance, the solution to the classical TSP is obtained when minimizing the length of the cycle, which is equivalent to the $\ell$-\texttt{mean} solution. However, when minimizing the maximum ($\ell$-\texttt{max}) length, the resulting cycles differ, and can be seen as robust solutions to the classical TSP. The solutions obtained by minimizing the absolute and square deviations with respect to the previous $\ell$-measures also differ. We would like to highlight some of the solutions compared with the classical TSP solution ($\ell$-\texttt{mean}).  In the $\ell$-\texttt{max} solution, in order to avoid large cycle lengths, the arc $(6,7)$ is not used, and construct a larger cycle but with shortest legs. The solution of the square deviation with respect to the maximum also seems to force the maximum length to be smaller, and the length of the legs in the cycle to be similar.

\subsection{Weighted Multicover Set Problem}

In the Weighted Multicover Set Problem (WMCSP) we are given a set of  $m$  items and  $n$ subsets of these items, $S_1, \ldots, S_n \subset N=\{1, \ldots, n\}$ are given. Each item $j \in M= \{1, \ldots, m\}$ is endowed with a weight $w_j$ that represents the preference of the item of its unit cost. The goal of the WMCSP is to decide which sets to accept to cover all the items by minimizing the weighted sum of the covered items. 
$$
y_i = \begin{cases}
    1 & \mbox{if subset $S_i$ is selected,}\\
    0 & \mbox{otherwise,}
\end{cases}
$$
the number of items of type $j$ can be calculated as $\sum_{i \in N:\atop j \in S_i} y_i$, that is, the number of times an activated set contains the item $j$. Thus, the overall cost/weight of item $j \in M$ for a solution induced by the $y$-variables is $x_j = w_j \sum_{i \in N:\atop j \in S_i} y_i$.

This problem has applications in various fields, such as healthcare, where items represent patients and sets correspond to different forms of medical care (e.g., nurses, physicians, medications, and medical assistance). Each patient is assigned a weight representing the severity of their illness or their level of urgency. A different application comes from telecommunication networks (as internet service providers or Wi-Fi routers to serve multiple users), were each service is allowed to \emph{cover} a set of users, each of them with a different demand that can be supplied by different types of connections (activated sets). 

The usual aggregation of the costs/weights is done using either the total sum or the average, which is then minimized to get a solution of the problem. Nevertheless, one can aggregate these costs using the measures described in this paper.

The constraints of the WMSCP are the following:
\begin{align*}
\sum_{i\in N:\atop j \in S_i} y_i \geq 1, \forall j \in M.
\end{align*}
to assure that all the items are covered by at least one set.

To show the differences between some of the approaches, we consider a simple instance of the problem, with $m=6$ items and $n=10$ sets that we draw in Figure \ref{fig:scp0}, where the sets are drawn at the top of the plot, items in the bottom (along with their weight), and the lines indicate the inclusion of the item in the set.

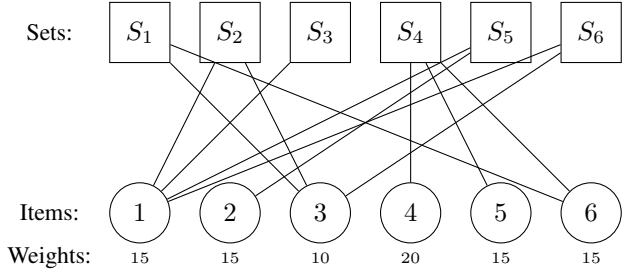
\begin{figure}[h!]
\begin{center}
\begin{tikzpicture}[scale=1.2, every node/.style={minimum size=0.8cm, inner sep=2pt}]
\node(sets) at (-3.5,2) {\small Sets:};
\node(sets) at (-3.5,0) {\small Items:};
\node(sets) at (-3.5,-0.5) {\small Weights:};
\node[draw, rectangle,thin, fill=white] (S0) at (-2.5, 2) {$S_{1}$};
\node[draw, rectangle,thin, fill=white] (S1) at (-1.5, 2) {$S_{2}$};
\node[draw, rectangle,thin, fill=white] (S2) at (-0.5, 2) {$S_{3}$};
\node[draw, rectangle,thin, fill=white] (S3) at (0.5, 2) {$S_{4}$};
\node[draw, rectangle,thin, fill=white] (S4) at (1.5, 2) {$S_{5}$};
\node[draw, rectangle,thin, fill=white] (S5) at (2.5, 2) {$S_{6}$};
\node[draw,circle] (I0) at (-2.5, 0) {$1$};
\node (w0) at (-2.5, -0.5) {\tiny $15$};
\node[draw,circle] (I1) at (-1.5, 0) {$2$};
\node (w1) at (-1.5, -0.5) {\tiny $15$};
\node[draw,circle] (I2) at (-0.5, 0) {$3$};
\node (w2) at (-0.5, -0.5) {\tiny $10$};
\node[draw,circle] (I3) at (0.5, 0) {$4$};
\node (w3) at (0.5, -0.5) {\tiny $20$};
\node[draw,circle] (I4) at (1.5, 0) {$5$};
\node (w4) at (1.5, -0.5) {\tiny $15$};
\node[draw,circle] (I5) at (2.5, 0) {$6$};
\node (w5) at (2.5, -0.5) {\tiny $15$};
\draw[thin] (S0) -- (I2);
\draw[thin] (S0) -- (I5);
\draw[thin] (S1) -- (I0);
\draw[thin] (S1) -- (I2);
\draw[thin] (S2) -- (I0);
\draw[thin] (S3) -- (I3);
\draw[thin] (S3) -- (I4);
\draw[thin] (S3) -- (I5);
\draw[thin] (S4) -- (I0);
\draw[thin] (S4) -- (I1);
\draw[thin] (S5) -- (I0);
\draw[thin] (S5) -- (I2);
\end{tikzpicture}
\end{center}
\caption{Instance for the WMCSP.\label{fig:scp0}}
\end{figure}

We compute the solution of the WMCSP for this instance, the $\ell$, $\mathfrak{b}$, and $\QO$ measures, with $\blambda$ induced by \texttt{mean}, $25\%$-$25\%$-trimmed \texttt{trimmed-mean}, and third quartile (\texttt{3quartile}), and the matrix $M$ is inducedinduced by the same $\blambda$ for the squared loss (Example \ref{ex:2}).
The $9$ solutions obtained with these methodologies are drawn in Figure \ref{fig:scp_sols}. We highlight in color gray the sets activated in each solution.
\begin{figure}[h!]
\begin{center}
\fbox{\adjustbox{scale=0.6}{\begin{tikzpicture}[scale=1.5, every node/.style={minimum size=0.8cm, inner sep=2pt}]
\node[draw, rectangle,thin, fill=white] (S0) at (-2.5, 2) {$S_{1}$};
\node[draw, rectangle,thin, fill=white] (S1) at (-1.5, 2) {$S_{2}$};
\node[draw, rectangle,thin, fill=white] (S2) at (-0.5, 2) {$S_{3}$};
\node[draw, rectangle,very thick, fill=gray!30] (S3) at (0.5, 2) {$S_{4}$};
\node[draw, rectangle,very thick, fill=gray!30] (S4) at (1.5, 2) {$S_{5}$};
\node[draw, rectangle,very thick, fill=gray!30] (S5) at (2.5, 2) {$S_{6}$};
\node[draw,circle] (I0) at (-2.5, 0) {$1$};
\node (w0) at (-2.5, -0.5) {\tiny $15$};
\node[draw,circle] (I1) at (-1.5, 0) {$2$};
\node (w1) at (-1.5, -0.5) {\tiny $15$};
\node[draw,circle] (I2) at (-0.5, 0) {$3$};
\node (w2) at (-0.5, -0.5) {\tiny $10$};
\node[draw,circle] (I3) at (0.5, 0) {$4$};
\node (w3) at (0.5, -0.5) {\tiny $20$};
\node[draw,circle] (I4) at (1.5, 0) {$5$};
\node (w4) at (1.5, -0.5) {\tiny $15$};
\node[draw,circle] (I5) at (2.5, 0) {$6$};
\node (w5) at (2.5, -0.5) {\tiny $15$};
\draw[thin] (S0) -- (I2);
\draw[thin] (S0) -- (I5);
\draw[thin] (S1) -- (I0);
\draw[thin] (S1) -- (I2);
\draw[thin] (S2) -- (I0);
\draw[very thick] (S3) -- (I3);
\draw[very thick] (S3) -- (I4);
\draw[very thick] (S3) -- (I5);
\draw[very thick] (S4) -- (I0);
\draw[very thick] (S4) -- (I1);
\draw[very thick] (S5) -- (I0);
\draw[very thick] (S5) -- (I2);
\end{tikzpicture}}}~\fbox{\adjustbox{scale=0.6}{\begin{tikzpicture}[scale=1.5, every node/.style={minimum size=0.8cm, inner sep=2pt}]
\node[draw, rectangle,thin, fill=white] (S0) at (-2.5, 2) {$S_{1}$};
\node[draw, rectangle,thin, fill=white] (S1) at (-1.5, 2) {$S_{2}$};
\node[draw, rectangle,very thick, fill=gray!30] (S2) at (-0.5, 2) {$S_{3}$};
\node[draw, rectangle,very thick, fill=gray!30] (S3) at (0.5, 2) {$S_{4}$};
\node[draw, rectangle,very thick, fill=gray!30] (S4) at (1.5, 2) {$S_{5}$};
\node[draw, rectangle,very thick, fill=gray!30] (S5) at (2.5, 2) {$S_{6}$};
\node[draw,circle] (I0) at (-2.5, 0) {$1$};
\node (w0) at (-2.5, -0.5) {\tiny $15$};
\node[draw,circle] (I1) at (-1.5, 0) {$2$};
\node (w1) at (-1.5, -0.5) {\tiny $15$};
\node[draw,circle] (I2) at (-0.5, 0) {$3$};
\node (w2) at (-0.5, -0.5) {\tiny $10$};
\node[draw,circle] (I3) at (0.5, 0) {$4$};
\node (w3) at (0.5, -0.5) {\tiny $20$};
\node[draw,circle] (I4) at (1.5, 0) {$5$};
\node (w4) at (1.5, -0.5) {\tiny $15$};
\node[draw,circle] (I5) at (2.5, 0) {$6$};
\node (w5) at (2.5, -0.5) {\tiny $15$};
\draw[thin] (S0) -- (I2);
\draw[thin] (S0) -- (I5);
\draw[thin] (S1) -- (I0);
\draw[thin] (S1) -- (I2);
\draw[very thick] (S2) -- (I0);
\draw[very thick] (S3) -- (I3);
\draw[very thick] (S3) -- (I4);
\draw[very thick] (S3) -- (I5);
\draw[very thick] (S4) -- (I0);
\draw[very thick] (S4) -- (I1);
\draw[very thick] (S5) -- (I0);
\draw[very thick] (S5) -- (I2);
\end{tikzpicture}}}~\fbox{\adjustbox{scale=0.6}{\begin{tikzpicture}[scale=1.5, every node/.style={minimum size=0.8cm, inner sep=2pt}]
\node[draw, rectangle,thin, fill=white] (S0) at (-2.5, 2) {$S_{1}$};
\node[draw, rectangle,thin, fill=white] (S1) at (-1.5, 2) {$S_{2}$};
\node[draw, rectangle,thin, fill=white] (S2) at (-0.5, 2) {$S_{3}$};
\node[draw, rectangle,very thick, fill=gray!30] (S3) at (0.5, 2) {$S_{4}$};
\node[draw, rectangle,very thick, fill=gray!30] (S4) at (1.5, 2) {$S_{5}$};
\node[draw, rectangle,very thick, fill=gray!30] (S5) at (2.5, 2) {$S_{6}$};
\node[draw,circle] (I0) at (-2.5, 0) {$1$};
\node (w0) at (-2.5, -0.5) {\tiny $15$};
\node[draw,circle] (I1) at (-1.5, 0) {$2$};
\node (w1) at (-1.5, -0.5) {\tiny $15$};
\node[draw,circle] (I2) at (-0.5, 0) {$3$};
\node (w2) at (-0.5, -0.5) {\tiny $10$};
\node[draw,circle] (I3) at (0.5, 0) {$4$};
\node (w3) at (0.5, -0.5) {\tiny $20$};
\node[draw,circle] (I4) at (1.5, 0) {$5$};
\node (w4) at (1.5, -0.5) {\tiny $15$};
\node[draw,circle] (I5) at (2.5, 0) {$6$};
\node (w5) at (2.5, -0.5) {\tiny $15$};
\draw[thin] (S0) -- (I2);
\draw[thin] (S0) -- (I5);
\draw[thin] (S1) -- (I0);
\draw[thin] (S1) -- (I2);
\draw[thin] (S2) -- (I0);
\draw[very thick] (S3) -- (I3);
\draw[very thick] (S3) -- (I4);
\draw[very thick] (S3) -- (I5);
\draw[very thick] (S4) -- (I0);
\draw[very thick] (S4) -- (I1);
\draw[very thick] (S5) -- (I0);
\draw[very thick] (S5) -- (I2);
\end{tikzpicture}}}\\
 \fbox{\adjustbox{scale=0.6}{\begin{tikzpicture}[scale=1.5, every node/.style={minimum size=0.8cm, inner sep=2pt}]
\node[draw, rectangle,thin, fill=white] (S0) at (-2.5, 2) {$S_{1}$};
\node[draw, rectangle,very thick, fill=gray!30] (S1) at (-1.5, 2) {$S_{2}$};
\node[draw, rectangle,thin, fill=white] (S2) at (-0.5, 2) {$S_{3}$};
\node[draw, rectangle,very thick, fill=gray!30] (S3) at (0.5, 2) {$S_{4}$};
\node[draw, rectangle,very thick, fill=gray!30] (S4) at (1.5, 2) {$S_{5}$};
\node[draw, rectangle,thin, fill=white] (S5) at (2.5, 2) {$S_{6}$};
\node[draw,circle] (I0) at (-2.5, 0) {$1$};
\node (w0) at (-2.5, -0.5) {\tiny $15$};
\node[draw,circle] (I1) at (-1.5, 0) {$2$};
\node (w1) at (-1.5, -0.5) {\tiny $15$};
\node[draw,circle] (I2) at (-0.5, 0) {$3$};
\node (w2) at (-0.5, -0.5) {\tiny $10$};
\node[draw,circle] (I3) at (0.5, 0) {$4$};
\node (w3) at (0.5, -0.5) {\tiny $20$};
\node[draw,circle] (I4) at (1.5, 0) {$5$};
\node (w4) at (1.5, -0.5) {\tiny $15$};
\node[draw,circle] (I5) at (2.5, 0) {$6$};
\node (w5) at (2.5, -0.5) {\tiny $15$};
\draw[thin] (S0) -- (I2);
\draw[thin] (S0) -- (I5);
\draw[very thick] (S1) -- (I0);
\draw[very thick] (S1) -- (I2);
\draw[thin] (S2) -- (I0);
\draw[very thick] (S3) -- (I3);
\draw[very thick] (S3) -- (I4);
\draw[very thick] (S3) -- (I5);
\draw[very thick] (S4) -- (I0);
\draw[very thick] (S4) -- (I1);
\draw[thin] (S5) -- (I0);
\draw[thin] (S5) -- (I2);
\end{tikzpicture}}}~\fbox{\adjustbox{scale=0.6}{\begin{tikzpicture}[scale=1.5, every node/.style={minimum size=0.8cm, inner sep=2pt}]
\node[draw, rectangle,very thick, fill=gray!30] (S0) at (-2.5, 2) {$S_{1}$};
\node[draw, rectangle,thin, fill=white] (S1) at (-1.5, 2) {$S_{2}$};
\node[draw, rectangle,very thick, fill=gray!30] (S2) at (-0.5, 2) {$S_{3}$};
\node[draw, rectangle,very thick, fill=gray!30] (S3) at (0.5, 2) {$S_{4}$};
\node[draw, rectangle,very thick, fill=gray!30] (S4) at (1.5, 2) {$S_{5}$};
\node[draw, rectangle,thin, fill=white] (S5) at (2.5, 2) {$S_{6}$};
\node[draw,circle] (I0) at (-2.5, 0) {$1$};
\node (w0) at (-2.5, -0.5) {\tiny $15$};
\node[draw,circle] (I1) at (-1.5, 0) {$2$};
\node (w1) at (-1.5, -0.5) {\tiny $15$};
\node[draw,circle] (I2) at (-0.5, 0) {$3$};
\node (w2) at (-0.5, -0.5) {\tiny $10$};
\node[draw,circle] (I3) at (0.5, 0) {$4$};
\node (w3) at (0.5, -0.5) {\tiny $20$};
\node[draw,circle] (I4) at (1.5, 0) {$5$};
\node (w4) at (1.5, -0.5) {\tiny $15$};
\node[draw,circle] (I5) at (2.5, 0) {$6$};
\node (w5) at (2.5, -0.5) {\tiny $15$};
\draw[very thick] (S0) -- (I2);
\draw[very thick] (S0) -- (I5);
\draw[thin] (S1) -- (I0);
\draw[thin] (S1) -- (I2);
\draw[very thick] (S2) -- (I0);
\draw[very thick] (S3) -- (I3);
\draw[very thick] (S3) -- (I4);
\draw[very thick] (S3) -- (I5);
\draw[very thick] (S4) -- (I0);
\draw[very thick] (S4) -- (I1);
\draw[thin] (S5) -- (I0);
\draw[thin] (S5) -- (I2);
\end{tikzpicture}}}~\fbox{\adjustbox{scale=0.6}{\begin{tikzpicture}[scale=1.5, every node/.style={minimum size=0.8cm, inner sep=2pt}]
\node[draw, rectangle,very thick, fill=gray!30] (S0) at (-2.5, 2) {$S_{1}$};
\node[draw, rectangle,thin, fill=white] (S1) at (-1.5, 2) {$S_{2}$};
\node[draw, rectangle,very thick, fill=gray!30] (S2) at (-0.5, 2) {$S_{3}$};
\node[draw, rectangle,very thick, fill=gray!30] (S3) at (0.5, 2) {$S_{4}$};
\node[draw, rectangle,very thick, fill=gray!30] (S4) at (1.5, 2) {$S_{5}$};
\node[draw, rectangle,thin, fill=white] (S5) at (2.5, 2) {$S_{6}$};
\node[draw,circle] (I0) at (-2.5, 0) {$1$};
\node (w0) at (-2.5, -0.5) {\tiny $15$};
\node[draw,circle] (I1) at (-1.5, 0) {$2$};
\node (w1) at (-1.5, -0.5) {\tiny $15$};
\node[draw,circle] (I2) at (-0.5, 0) {$3$};
\node (w2) at (-0.5, -0.5) {\tiny $10$};
\node[draw,circle] (I3) at (0.5, 0) {$4$};
\node (w3) at (0.5, -0.5) {\tiny $20$};
\node[draw,circle] (I4) at (1.5, 0) {$5$};
\node (w4) at (1.5, -0.5) {\tiny $15$};
\node[draw,circle] (I5) at (2.5, 0) {$6$};
\node (w5) at (2.5, -0.5) {\tiny $15$};
\draw[very thick] (S0) -- (I2);
\draw[very thick] (S0) -- (I5);
\draw[thin] (S1) -- (I0);
\draw[thin] (S1) -- (I2);
\draw[very thick] (S2) -- (I0);
\draw[very thick] (S3) -- (I3);
\draw[very thick] (S3) -- (I4);
\draw[very thick] (S3) -- (I5);
\draw[very thick] (S4) -- (I0);
\draw[very thick] (S4) -- (I1);
\draw[thin] (S5) -- (I0);
\draw[thin] (S5) -- (I2);
\end{tikzpicture}}}\\
\fbox{\adjustbox{scale=0.6}{\begin{tikzpicture}[scale=1.5, every node/.style={minimum size=0.8cm, inner sep=2pt}]
\node[draw, rectangle,very thick, fill=gray!30] (S0) at (-2.5, 2) {$S_{1}$};
\node[draw, rectangle,thin, fill=white] (S1) at (-1.5, 2) {$S_{2}$};
\node[draw, rectangle,thin, fill=white] (S2) at (-0.5, 2) {$S_{3}$};
\node[draw, rectangle,very thick, fill=gray!30] (S3) at (0.5, 2) {$S_{4}$};
\node[draw, rectangle,very thick, fill=gray!30] (S4) at (1.5, 2) {$S_{5}$};
\node[draw, rectangle,thin, fill=white] (S5) at (2.5, 2) {$S_{6}$};
\node[draw,circle] (I0) at (-2.5, 0) {$1$};
\node (w0) at (-2.5, -0.5) {\tiny $15$};
\node[draw,circle] (I1) at (-1.5, 0) {$2$};
\node (w1) at (-1.5, -0.5) {\tiny $15$};
\node[draw,circle] (I2) at (-0.5, 0) {$3$};
\node (w2) at (-0.5, -0.5) {\tiny $10$};
\node[draw,circle] (I3) at (0.5, 0) {$4$};
\node (w3) at (0.5, -0.5) {\tiny $20$};
\node[draw,circle] (I4) at (1.5, 0) {$5$};
\node (w4) at (1.5, -0.5) {\tiny $15$};
\node[draw,circle] (I5) at (2.5, 0) {$6$};
\node (w5) at (2.5, -0.5) {\tiny $15$};
\draw[very thick] (S0) -- (I2);
\draw[very thick] (S0) -- (I5);
\draw[thin] (S1) -- (I0);
\draw[thin] (S1) -- (I2);
\draw[thin] (S2) -- (I0);
\draw[very thick] (S3) -- (I3);
\draw[very thick] (S3) -- (I4);
\draw[very thick] (S3) -- (I5);
\draw[very thick] (S4) -- (I0);
\draw[very thick] (S4) -- (I1);
\draw[thin] (S5) -- (I0);
\draw[thin] (S5) -- (I2);
\end{tikzpicture}}}~\fbox{\adjustbox{scale=0.6}{\begin{tikzpicture}[scale=1.5, every node/.style={minimum size=0.8cm, inner sep=2pt}]
\node[draw, rectangle,very thick, fill=gray!30] (S0) at (-2.5, 2) {$S_{1}$};
\node[draw, rectangle,thin, fill=white] (S1) at (-1.5, 2) {$S_{2}$};
\node[draw, rectangle,thin, fill=white] (S2) at (-0.5, 2) {$S_{3}$};
\node[draw, rectangle,very thick, fill=gray!30] (S3) at (0.5, 2) {$S_{4}$};
\node[draw, rectangle,very thick, fill=gray!30] (S4) at (1.5, 2) {$S_{5}$};
\node[draw, rectangle,thin, fill=white] (S5) at (2.5, 2) {$S_{6}$};
\node[draw,circle] (I0) at (-2.5, 0) {$1$};
\node (w0) at (-2.5, -0.5) {\tiny $15$};
\node[draw,circle] (I1) at (-1.5, 0) {$2$};
\node (w1) at (-1.5, -0.5) {\tiny $15$};
\node[draw,circle] (I2) at (-0.5, 0) {$3$};
\node (w2) at (-0.5, -0.5) {\tiny $10$};
\node[draw,circle] (I3) at (0.5, 0) {$4$};
\node (w3) at (0.5, -0.5) {\tiny $20$};
\node[draw,circle] (I4) at (1.5, 0) {$5$};
\node (w4) at (1.5, -0.5) {\tiny $15$};
\node[draw,circle] (I5) at (2.5, 0) {$6$};
\node (w5) at (2.5, -0.5) {\tiny $15$};
\draw[very thick] (S0) -- (I2);
\draw[very thick] (S0) -- (I5);
\draw[thin] (S1) -- (I0);
\draw[thin] (S1) -- (I2);
\draw[thin] (S2) -- (I0);
\draw[very thick] (S3) -- (I3);
\draw[very thick] (S3) -- (I4);
\draw[very thick] (S3) -- (I5);
\draw[very thick] (S4) -- (I0);
\draw[very thick] (S4) -- (I1);
\draw[thin] (S5) -- (I0);
\draw[thin] (S5) -- (I2);
\end{tikzpicture}}}~\fbox{\adjustbox{scale=0.6}{\begin{tikzpicture}[scale=1.5, every node/.style={minimum size=0.8cm, inner sep=2pt}]
\node[draw, rectangle,thin, fill=white] (S0) at (-2.5, 2) {$S_{1}$};
\node[draw, rectangle,very thick, fill=gray!30] (S1) at (-1.5, 2) {$S_{2}$};
\node[draw, rectangle,thin, fill=white] (S2) at (-0.5, 2) {$S_{3}$};
\node[draw, rectangle,very thick, fill=gray!30] (S3) at (0.5, 2) {$S_{4}$};
\node[draw, rectangle,very thick, fill=gray!30] (S4) at (1.5, 2) {$S_{5}$};
\node[draw, rectangle,thin, fill=white] (S5) at (2.5, 2) {$S_{6}$};
\node[draw,circle] (I0) at (-2.5, 0) {$1$};
\node (w0) at (-2.5, -0.5) {\tiny $15$};
\node[draw,circle] (I1) at (-1.5, 0) {$2$};
\node (w1) at (-1.5, -0.5) {\tiny $15$};
\node[draw,circle] (I2) at (-0.5, 0) {$3$};
\node (w2) at (-0.5, -0.5) {\tiny $10$};
\node[draw,circle] (I3) at (0.5, 0) {$4$};
\node (w3) at (0.5, -0.5) {\tiny $20$};
\node[draw,circle] (I4) at (1.5, 0) {$5$};
\node (w4) at (1.5, -0.5) {\tiny $15$};
\node[draw,circle] (I5) at (2.5, 0) {$6$};
\node (w5) at (2.5, -0.5) {\tiny $15$};
\draw[thin] (S0) -- (I2);
\draw[thin] (S0) -- (I5);
\draw[very thick] (S1) -- (I0);
\draw[very thick] (S1) -- (I2);
\draw[thin] (S2) -- (I0);
\draw[very thick] (S3) -- (I3);
\draw[very thick] (S3) -- (I4);
\draw[very thick] (S3) -- (I5);
\draw[very thick] (S4) -- (I0);
\draw[very thick] (S4) -- (I1);
\draw[thin] (S5) -- (I0);
\draw[thin] (S5) -- (I2);
\end{tikzpicture}}}
\end{center}
\caption{Solutions obtained with different approaches for the WMCSP.\label{fig:scp_sols}}
\end{figure}
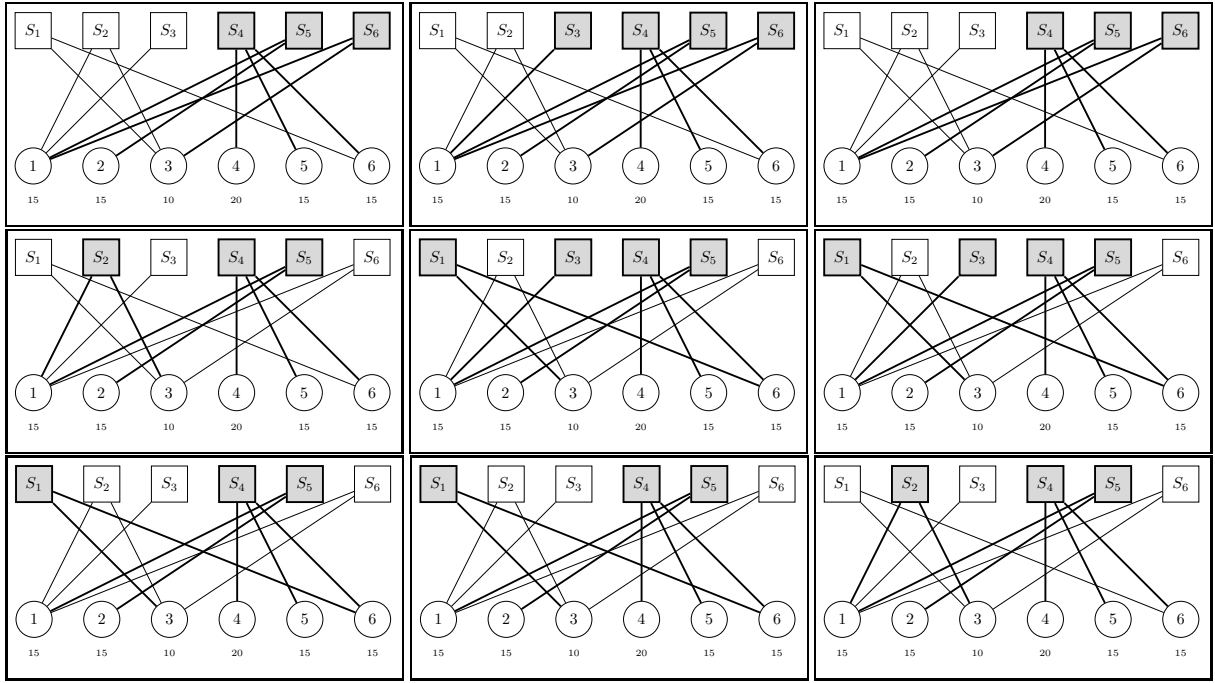


Note that, even in this toy example, designed for illustrative purposes, the solutions obtained with different methodologies are different, each of them responding to the minimization of the measure under study. In this way, decision makers are provided with an unified tool to construct feasible solutions of the same problem, but guided through different objectives.

\section{Conclusions}

In this paper, we provide a novel mathematical optimization-based framework for calculating measures that generalize a wide range of metrics that commonly appear in the Operations Research and Data Science literature. Apart from developing linear and linear bilevel formulations for computing $\tL$ and nested ordered measures, we propose a new family of quadratic  ordered-type measures, and develop a mixed integer linear optimization model for its computations. We empirically test our approaches  on a wide variety of instances, reporting the computation times required to calculate these measures, as well as their accuracy. With this study, we conclude that the optimization models together with the algorithmic strategies  implemented in off-the-shelf optimization solvers,  exhibit a similar performance than the add-hoc python libraries developed  to this end. Finally, we show how to incorporate these measures into optimization problems, where the values to be summarized with these measures are now part of the decision process, and then unknown. We show some results applying these models to ordered versions of well-known continuous and combinatorial  optimization problems,  namely, scenario analysis linear programming,  traveling salesman  and set cover problems,  although similarly they can be applied to other optimization problems, both with discrete or continuous domain. In most  cases, changing the metric that summarizes the values results in different solution structures.

\textbf{Acknowledgements}

The authors acknowledge financial support by the reasearch project PID2020-114594GB-C21 funded by the Spanish Ministerio de Ciencia e Innovación and Agencia Estatal de Investigación (MCIN/AEI/10.13039/501100011033).

\bibliographystyle{acm}  
\bibliography{refs}

\end{document}